\documentclass[12pt]{amsart}

\usepackage{amssymb, tikz}

\usepackage{hyperref}

\usepackage{amsmath}

\newtheorem{theorem}{Theorem}[section]
\newtheorem{observation}{Observation}[section]
\newtheorem{lemma}[theorem]{Lemma}

\newtheorem{fact}[theorem]{Fact}

\newtheorem{corollary}[theorem]{Corollary}
\newtheorem{claim}[theorem]{Claim}

\theoremstyle{definition}
\newtheorem{definition}[theorem]{Definition}

\newtheorem{remark}[theorem]{Remark}

\let \restr = \upharpoonright

\let \into = \longrightarrow

\let \sub = \subseteq
\let \elsub = \preccurlyeq

\let \av = \arrowvert

\let \a = \alpha
\let \b = \beta
\let \g = \gamma
\let \d = \delta

\let \l = \lambda
\let \k = \kappa

\let \n = \nu
\let \p = \pi
\let \r = \rho

\let \D = \Delta

\let \x = \xi

\let \o = \omega

\let \al = \aleph
\let \la = \langle
\let \ra = \rangle
\let \mtcl = \mathcal
\let \mtbb = \mathbb

\DeclareMathOperator{\dom}{dom}

\DeclareMathOperator{\cf}{cf}

\DeclareMathOperator{\cl}{cl}

\DeclareMathOperator{\Lim}{Lim}
\DeclareMathOperator{\Succ}{Succ}
\DeclareMathOperator{\Ord}{\textsf{Ord}}

\DeclareMathOperator{\ZFC}{ZFC}
\DeclareMathOperator{\ZF}{ZF}

\DeclareMathOperator{\PFA}{\textsf{PFA}}
\DeclareMathOperator{\BPFA}{\textsf{BPFA}}
\DeclareMathOperator{\FA}{\textsf{FA}}
\DeclareMathOperator{\CH}{\textsf{CH}}
\DeclareMathOperator{\NS}{\textsf{NS}}
\DeclareMathOperator{\WCG}{\textsf{WCG}}
\DeclareMathOperator{\GCH}{\textsf{GCH}}
\DeclareMathOperator{\M}{\textsf{Meas}}
\DeclareMathOperator{\Measuring}{\textsf{Measuring}}
\DeclareMathOperator{\Unif}{\textsf{Unif}}

\title{Few new reals}

\author[D. Asper\'o]{David Asper\'o}

\address{David Asper\'o, School of Mathematics, University of East Anglia, Norwich NR4 7TJ, UK}

\email{d.aspero@uea.ac.uk}

\author[M.A. Mota]{Miguel Angel Mota}

\address{Miguel Angel Mota,  Departamento de Matem\'aticas,
ITAM,
01080, Mexico City, Mexico}

\email{motagaytan@gmail.com}

\thanks{The first author acknowledges support of EPSRC Grant EP/N032160/1.}
\thanks{We thank Tanmay Inamdar, Tadatoshi Miyamoto, Justin Moore, Itay Neeman, Boban Veli\v{c}kovi\'{c}, and Teruyuki Yorioka for their input at various stages in the creation of this article. We also thank the anonymous referee for their careful reading of the paper and their useful comments.}

\date{}

\begin{document}


\maketitle
\pagestyle{myheadings}\markright{Few new reals}

\begin{abstract}
We introduce a new method for building models of $\CH$, together with $\Pi_2$ statements over $H(\omega_2)$, by forcing. Unlike other forcing constructions in the literature, our construction adds new reals, although only  $\aleph_1$-many of them. Using this approach, we build a model in which a very strong form of the negation of Club Guessing at $\omega_1$ known as $\Measuring$ holds together with $\CH$, thereby answering a well-known question of Moore. This construction can be described as a finite-support weak forcing iteration with side conditions consisting of suitable graphs of sets of models with markers.
The $\CH$-preservation is accomplished through the imposition of copying constraints on the information carried by the condition, as dictated by the edges in the graph.
\end{abstract}

\section{Introduction}
The problem of building models of consequences, at the level of $H(\omega_2)$, of classical forcing axioms in the presence of the Continuum Hypothesis ($\CH$) has a long history, starting with Jensen's landmark result  that Suslin's Hypothesis is compatible with $\CH$ (\cite{Devlin-Johnsbraten}). Much of the work in this area is due to Shelah (see \cite{PIF}), with contributions also by other people (see e.g.\  \cite{Abraham-Todorcevic}, \cite{Eisworth-Roitman}, \cite{M0}, \cite{EMM}, \cite{ALM} or \cite{M}). Most of the work in the area done so far proceeds by showing that some suitable countable support iteration whose iterands are proper forcing notions not adding new reals fails to add new reals also at limit stages.

There are (nontrivial) limitations to what can be achieved in this area. One conclusive example is the main result from \cite{ALM}, which highlights a strong global limitation: There is no model of $\CH$ satisfying a certain mild large cardinal assumption and realizing all $\Pi_2$ statements over the structure $H(\omega_2)$ that can be forced, using proper forcing, to hold together with $\CH$. In fact there are two $\Pi_2$ statements over $H(\omega_2)$, each of which can be forced, using proper forcing, to hold together with $\CH$---for one of them we need an inaccessible limit of measurable cardinals---and whose conjunction implies $2^{\aleph_0}=2^{\aleph_1}$.

The above example is closely tied to the following well-known ob\-sta\-cle to not adding reals, which appears in \cite{Devlin-Shelah} (s.\ also \cite{EMM}) and which is more to the point in the context of this paper:\footnote{We will revisit this obstacle at Subsection 2.2 with the purpose of addressing the following question: why do our methods work with the present application (forcing $\Measuring$) and not with the problem of forcing $\Unif(\vec C)$ (for any given $\vec C$)?} Given a ladder system $\vec C = (C_\delta\,:\,\delta\in \Lim(\omega_1))$ (i.e., each $C_\delta$ is a cofinal subset of $\delta$ of order type $\omega$), let $\Unif(\vec C)$ denote the statement that for every colouring $F:\Lim(\omega_1)\longrightarrow \{0, 1\}$ there is a function $G:\omega_1\longrightarrow   \{0, 1\}$ with the property that for every $\delta\in \Lim(\omega_1)$ there is some $\alpha<\delta$ such that $G(\xi)=F(\delta)$ for all $\xi\in C_\delta\setminus\alpha$ (where, given an ordinal $\alpha$, $\Lim(\alpha)$ is the set of limit ordinals below $\alpha$). We say that $G$ uniformizes $F$ on $\vec C$. Given $\vec C$ and $F$ as above there is a natural forcing notion, let us call it $\mathcal Q_{\vec C, F}$, for adding a uniformizing function for $F$ on $\vec C$ by initial segments. It takes a standard exercise to show that $\mathcal Q_{\vec C, F}$ is proper, adds the intended uniformizing function, and does not add reals. However, any long enough iteration of forcings of the form $\mathcal Q_{\vec C, F}$, even with a fixed $\vec C$, will necessarily add new reals. As a matter of fact, the existence of a ladder system $\vec C$ for which $\Unif(\vec C)$ holds cannot be forced together with $\CH$ in any way whatsoever, as this statement actually implies $2^{\aleph_0}=2^{\aleph_1}$. The argument is well-known and may be found for example in \cite{Devlin-Shelah} and in \cite{EMM}.

In the present paper we distance ourselves from the tradition of i\-te\-ra\-ting forcing without adding reals and tackle the problem of building interesting models of $\CH$ with an entirely different approach: starting with a model of $\CH$, we build a forcing which adds new reals,\footnote{As it turns out, the construction resembles a classical finite support iteration, and in fact it adds Cohen reals.} albeit only $\aleph_1$-many of them.

In \cite{forcing-conseqs}, a framework for building finite support forcing iterations incorporating systems of countable models as side conditions was developed (see also \cite{cons-club-guessing-failure},  \cite{genMA},  \cite{separating-fat} for further elaborations). These iterations arise naturally in, for example, situations in which one is in\-te\-res\-ted in building a forcing iteration of length $\kappa$ (where $\kappa$ is relatively long) which is proper and which, in addition, does not collapse cardinals.\footnote{For example if, as in \cite{forcing-conseqs}, we want to force certain instances of the Proper Forcing Axiom ($\PFA$) together with $2^{\aleph_0}=\kappa>\aleph_2$.} Much of what we will say in the next few paragraphs will probably make sense  only to readers with at least some familiarity with the framework as presented, for example, in \cite{forcing-conseqs}.

In the situations we are referring to here, one typically aims at a construction which in fact has the $\aleph_2$-chain condition, and in order to achieve this goal it is natural to build the iteration in such a way that conditions be of the form $(F, \Delta)$, for $F$ a (finitely supported) $\kappa$-sequence of working parts, and with $\Delta$ being a set of models with markers, i.e., a set of ordered pairs $(N, \rho)$, where $N$ is a countable elementary submodel of $H(\kappa)$, possibly enhanced with some predicate $T\subseteq H(\kappa)$, and where $\rho\in N\cap\kappa$. $N$ is one of the models for which we will try to `force' each working part $F(\alpha)$, for every stage $\alpha\in N\cap\rho$, to be generic for the generic extension of $N$ up to that stage; thus, $\rho$ is to be seen as a `marker' that tells us up to which point is $N$ to be seen as `active' as a side condition.

In order for the construction to have the $\aleph_2$-chain condition and be proper, it is often necessary to start from a model of $\CH$ and require that the domain of $\Delta$ be a set of models with suitable symmetry properties. We call (finite) sets of models having these properties $T$-symmetric systems (for a fixed $T\subseteq H(\kappa)$). One of these properties, and the one on which we will focus our attention in a moment, is the following: In a $T$-symmetric system $\mathcal N$, if $N$ and $N'$ are both in $\mathcal N$ and $N\cap\omega_1=N'\cap\omega_1$, then there is a (unique) isomorphism $\Psi_{N, N'}$ between the structures $(N; \in, T, \mtcl N\cap N)$ and $(N'; \in, T, \mtcl N\cap N')$ which, moreover, is the identity on $N\cap N'$.

At this point one could as well take a step back and analyse the pure side condition forcing $\mtcl P_0$ by itself. This forcing $\mtcl P_0$, which we can naturally see as the first stage of our construction, consists of all finite $T$-symmetric systems of submodels, ordered by reverse inclusion. $\mtcl P_0$ first appeared in the literature in \cite{Todor}. It is a relatively well-known fact,  and was noted in \cite{separating-fat},\footnote{See also \cite{KT}.} that forcing with $\mtcl P_0$ adds Cohen reals, although not too many; in fact it adds exactly $\aleph_1$-many of them. This may be somewhat surprising given that $\mtcl P_0$ adds, by finite approximations, a new rather large object (a symmetric system covering all of $H(\kappa)^V$).\footnote{Incidentally, $\mtcl P_0$ is in fact strongly proper, and so each new real it adds is in fact contained in an extension of $V$ by some Cohen real. The preservation of $\CH$ by $\mtcl P_0$ was exploited in \cite{Mota-Krueger}.} The argument for this is contained in the proof of Lemma \ref{fnr-0} from the present paper, but it will nonetheless be convenient at this point to sketch it here.

Let us assume, towards a contradiction, that $\CH$ holds and there is a sequence $(\dot r_\nu)_{\nu<\omega_2}$ of $\mtcl P_0$-names which some condition $\mtcl N$ forces to be distinct subsets of $\omega$. Without loss of generality we may take each $\dot r_\nu$ to be a member of $H(\kappa)$. For each $\nu$ we can pick $N_\nu$ to be a sufficiently correct countable model---meaning that $(N_\nu; \in, T^*)\elsub (H(\k); \in, T^*)$ for a suitably expressive predicate $T^*\sub H(\k)$---containing all relevant objects, which in this case includes $\mtcl N$ and  $\dot r_\nu$. As $\CH$ holds, we may find distinct indices $\nu$ and $\nu'$ such that there is a unique isomorphism $\Psi_{N_\nu, N_{\nu'}}$ between the structures $(N_\nu; \in, T^\ast, \mtcl N, \dot r_\nu)$ and $(N_{\nu'}; \in, T^\ast, \mtcl N, \dot r_{\nu'})$ fixing $N_\nu\cap N_{\nu'}$. But then $\mtcl N^\ast = \mtcl N\cup\{N_\nu, N_{\nu'}\}$ is a condition in $\mtcl P_0$ forcing that $\dot r_\nu = \dot r_{\nu'}$. The point is that if $n\in \omega$ and $\mtcl N'$ is any condition extending $\mtcl N^\ast$ and forcing $n\in \dot r_\nu$, then $\mtcl N'$ is in fact compatible with a condition $\mathcal M\in N_\nu$ forcing the same thing. This is true since $\mathcal N^\ast$ is an $(N_\nu, \mathcal P_0)$-generic condition. But then $\Psi_{N_\nu, N_{\nu'}}(\mathcal M)$ is a condition forcing $n\in \Psi_{N_\n, N_{\n'}}(\dot r_\n)=\dot r_{\n'}$ (since, by taking $T^*$ expressive enough, we may assume the forcing relation for $\mtcl P_0$ to be definable in $(H(\k); \in, T^\ast)$ without parameters). Finally, if $\mtcl N''$ is any common extension of $\mtcl N'$ and $\mtcl M$, then $\mathcal N''$ forces also that $n\in \dot r_{\n'}$, since it extends $\Psi_{N_\n, N_{\n'}}(\mathcal M)$ as $\Psi_{N_\n, N_{\n'}}(\mathcal M)\subseteq \mathcal N''$ by the symmetry requirement.\footnote{It is worth noticing the resemblance of this argument with Shelah's argument for showing that $\CH$ gets preserved by the limit of any countable support iteration of length less than $\o_2$ of proper forcings of size at most $\aleph_1$ (s.\ e.g.\ the proof of   \cite[Theorem 2.10]{abraham}.)}

$\mathcal P_0$ has received some attention in the literature. For example, Todor\-\v{c}e\-vi\'{c} proved that $\mtcl P_0$ adds a Kurepa tree (s.\ \cite{KT}). Also, \cite{KT} presents a mild variant of $\mtcl P_0$ which not only preserves $\CH$ but actually forces $\diamondsuit$.

The iterations with symmetric systems of models as side conditions that we were referring to before do not preserve $\CH$, and in fact they force $2^{\aleph_0}=\kappa>\aleph_1$. The reason is of course that there are no symmetry requirements on the working parts. Actually, even if the first stage of the iterations---which is, essentially, $\mtcl P_0$---preserves $\CH$, the iterations are in fact designed to add new reals at all later (successor) stages.

Something one may naturally envision at this point is the possibility to build a suitable forcing with systems of models (with mar\-kers) as side conditions while strengthening the symmetry constraints, so as to make them apply not only to the side condition part of the forcing but also to the working parts; one would hope to exploit the above idea in order to show that the forcing thus constructed preserves $\CH$, and would of course like to be able to do that while at the same time forcing some interesting statement. In the present paper we implement this idea by proving that a very strong form of the failure of Club Guessing at $\omega_1$ known as $\Measuring$ (see \cite{EMM}) that follows from $\PFA$ can be forced adding new reals while, nevertheless, preserving $\CH$.

\begin{definition} $\Measuring$ holds if and only if for every sequence $\vec C=(C_\delta\,:\,\delta\in \omega_1)$, if each $C_\delta$ is a closed subset of $\delta$ in the order topology, then there is a club $C\sub\omega_1$ such that for every $\delta\in C$ there is some $\alpha<\delta$ such that either

\begin{itemize}
\item $(C \cap\delta)\setminus\alpha\subseteq C_\delta$, or
\item $(C\setminus\alpha)\cap C_\delta=\emptyset$.
\end{itemize}
\end{definition}

In the above definition, we say that \emph{$C$ measures $\vec C$.}  $\Measuring$ is of course equivalent to its restriction to club-sequences $\vec C$ on $\omega_1$, i.e., to sequences of the form $\vec C=(C_\delta\,:\,\delta\in \Lim(\omega_1))$, where each $C_\delta$ is a club of  $\delta$. It is also not difficult to see that $\Measuring$ can be rephrased as the assertion that the algebra $\mtcl P(\o_1)/\NS_{\o_1}$---where $\NS_{\o_1}$ denotes the nonstationary ideal on $\o_1$---forces that $\mtcl C_{\o_1}^V$ is a base for an ultrafilter on the Boolean subalgebra of $\mtcl P(\o_1^V)$ generated by the closed sets as computed in the generic ultrapower $M=V/\dot G$, where $\mtcl C_{\o_1}^V$ denotes the club filter on $\o_1$ in $V$.

A partial order $\mtbb P$ is $\al_2$-Knaster if for every sequence $(q_\x\,:\,\x<\o_2)$ of $\mtbb P$-conditions there is a set $I\sub\o_2$ of cardinality $\al_2$ such that $q_\x$ and $q_{\x'}$ are compatible for all $\x$, $\x'\in I$. Of course, every $\al_2$-Knaster partial order has the $\al_2$-chain condition.

Our main theorem is the following.

\begin{theorem}\label{mainthm-0-intro} ($\CH$) Let $\kappa\geq\omega_2$ be a regular cardinal such that $2^{{<}\k}=\k$.
Then there is a partial order $\mtcl P\sub H(\k)$ with the following properties.
\begin{enumerate}
\item $\mtcl P$ is proper.
\item $\mtcl P$ is $\aleph_2$-Knaster.
\item $\mtcl P$ forces the following statements.
\begin{enumerate}
\item $\Measuring$
\item $\CH$
\item $2^{\aleph_1}=\k$
\end{enumerate}
\end{enumerate}
\end{theorem}

Theorem \ref{mainthm-0-intro} answers a question of Moore, who asked if $\Measuring$ is compatible with $\CH$ (see \cite{EMM} or \cite{Moore2}). The relative consistency of $\Measuring$ with $\CH$ has also been obtained recently by Golshani and Shelah in \cite{Golshani-Shelah}, where they have actually shown that every countable support iteration of the natural proper posets for adding a club of $\omega_1$ measuring a given club-sequence by countable approximations fails to add new reals.\footnote{It is straightforward to see that these natural forcings for adding a given instance of $\Measuring$ do not add reals; however, before \cite{Golshani-Shelah} it was not known whether their countable support iterations also (consistently) have this property.} Prior to \cite{Golshani-Shelah}, the strongest failures of Club Guessing at $\omega_1$ known to be within reach of the forcing iteration methods for producing models of $\CH$ without adding new reals (s.\ \cite{NNR}) were only in the region of the negation of weak Club Guessing at $\omega_1$, $\lnot\WCG$, which is the statement that for every ladder system $(C_\delta\,:\,\delta\in\Lim(\omega_1))$ there is a club $C\subseteq \omega_1$ having finite intersection with each $C_\delta$.\footnote{$\Measuring$  implies $\lnot\WCG$. To see this, suppose $(C_\delta\,:\,\delta\in\Lim(\o_1))$ is a ladder system and $D\subseteq\o_1$ is a club measuring it. Then every limit point $\d\in D$ of limit points of $D$ is such that $D\cap C_\delta$ is bounded in $\delta$ since no tail of $D\cap\delta$ can possibly be contained in $C_\delta$ as $C_\delta$ has order type only $\omega$.}  Moore, upon learning about an earlier version of Theorem \ref{mainthm-0-intro}, asked whether $\Measuring$ implies that there are non-constructible reals. This question was aimed at addressing the issue whether or not adding new reals is a ne\-ce\-ssary feature of any successful approach to forcing $\Measuring$ + $\CH$, and it obtains a negative answer by the Golshani-Shelah result.

Our construction is a sequence $\langle\mtcl P_\b\,:\,\b\leq\k\rangle$ which is not a forcing iteration, in the usual sense of $\mtcl P_\a$ being a complete suborder of $\mtcl P_\b$ for all $\a<\b\leq\k$, but which nevertheless has a sufficiently nice property; it is what we will refer to as a \emph{weak forcing iteration}. This means that for all $\a<\b$, every $\mtcl P_\a$-condition is a $\mtcl P_\b$-condition, for all $p_0$, $p_1\in\mtcl P_\a$, if $p_1\leq_{\mtcl P_\a} p_0$, then $p_1\leq_{\mtcl P_\b} p_0$,\footnote{Although it not be the case that if $p_1\leq_{\mtcl P_\b} p_0$, then $p_1\leq_{\mtcl P_\a} p_0$. In other words, $\mtcl P_\a$ need not be a suborder of $\mtcl P_\b$.} and, moreover, every predense subset of $\mtcl P_\a$ is also predense in $\mtcl P_\b$.
 Using this piece of terminology, our construction can be roughly described as a finitely supported weak forcing iteration $\langle \mtcl P_\b\,:\,\b\leq\k\rangle$ in which conditions come together with a side condition consisting of a graph of edges $\{(N_0, \r_0), (N_1, \r_1)\}$, where each $(N_i, \r_i)$ is a model with markers, with suitable structural properties. Given any such edge $\{(N_0, \r_0), (N_1, \r_1)\}$, $N_0\cong N_1$. Furthermore, all the information carried by the condition---including both its working part and its side condition---contained in $N_0$ and attached to any $\alpha\in N_0\cap\r_0$ such that $\Psi_{N_0, N_1}(\a)<\r_1$ (where $\Psi_{N_0, N_1}$ is the unique isomorphism between $(N_0; \in)$ and $(N_1; \in)$) is to be copied over into $N_1$ by $\Psi_{N_0, N_1}$. This copying will be crucially used in the proof of $\CH$-preservation\footnote{See also \cite{A-G} for another forcing construction using edges in order to preserve $\GCH$.} and also in other parts of the proof of Theorem \ref{mainthm-0-intro} (most notably in the proof of the $\al_2$-chain condition). The working part consists of conditions for natural forcing notions adding instances of $\Measuring$.

Rather than delving into more details here, we direct the reader to the actual construction in Section \ref{first-construction}.

\subsection{Some observations on extensions of $\Measuring$}

We conclude this introduction by briefly considering some extensions of $\Measuring$.

It is immediate to see that $\Measuring$ is equivalent to the statement that if $(\mathcal C_\delta\,:\,\delta\in\Lim(\omega_1))$ is such that each $\mathcal C_\delta$ is a countable collection of closed subsets of $\delta$, then there is a club of $\omega_1$ measuring all members of $\mathcal C_\delta$ for each $\delta$.
We may thus consider the following family of strengthenings of $\Measuring$.

\begin{definition} Given a cardinal $\kappa$, $\M_{\kappa}$ holds if and only if for every family $\mathcal C$ consisting of closed subsets of $\omega_1$ and such that $\arrowvert\mathcal C\arrowvert\leq\kappa$ there is a club $C\subseteq \omega_1$ with the property that for every $D\in\mathcal C$ and every $\delta\in C$ there is some $\alpha<\delta$ such that either

 \begin{itemize}
\item $(C \cap\delta)\setminus\alpha\subseteq D$, or
\item $((C\cap\d)\setminus\alpha)\cap D=\emptyset$.
\end{itemize}
\end{definition}

$\M_{\aleph_0}$ is trivially true in $\ZFC$. Also, it is clear that $\M_{\kappa}$ implies $\M_{\lambda}$ whenever $\lambda<\kappa$, and that $\M_{\aleph_1}$ implies $\Measuring$.

Recall that the splitting number, $\mathfrak s$, is the minimal cardinality of a splitting family, i.e., of a collection $\mathcal X\sub[\omega]^{\aleph_0}$ such that for every $Y\in[\omega]^{\aleph_0}$ there is some $X\in\mathcal X$ such that $X\cap Y$ and $Y\setminus X$ are both infinite.

In the proof of Fact \ref{fact-on-s}, if $(C_\d\,:\,\d\in\Lim(\o_1))$ is a ladder system on $\o_1$, we write $(C_\delta(n))_{n<\omega}$ to denote the strictly increasing enu\-me\-ration of $C_\delta$. Also, $[\alpha,\,\beta)=\{\xi\in \Ord\,:\,\alpha\leq\xi<\beta\}$ for all ordinals $\alpha\leq \beta$.

\begin{fact}\label{fact-on-s} $\M_{\mathfrak s}$ is false.
\end{fact}

\begin{proof}
Let $\mathcal X\subseteq[\omega]^{\aleph_0}$ be a splitting family. Let $(C_\delta)_{\delta\in \Lim(\omega)}$ be a ladder system on $\omega_1$ such that $C_\d(n)$ is a successor ordinal for each $\d\in\Lim(\o_1)$ and $n<\o$, and let $\mathcal C$ be the collection of all sets of the form $$Z^X_\delta=\bigcup \{[C_\delta(n),\,C_\delta(n+1))\,:\,n\in X\}\cup\{\delta\}$$ for some $\delta\in\Lim(\omega_1)$ and $X\in\mathcal X$.  Let $D$ be a club of $\omega_1$, let $\delta<\omega_1$ be a limit point of $D$, and let $$Y=\{n<\omega\,:\,[C_\delta(n),\,C_\delta(n+1))\cap D\neq\emptyset\}$$ Let $X\in\mathcal X$ be such that $X\cap Y$ and $Y\setminus X$ are infinite. Then $Z^X_\delta\cap D$ and $D\setminus Z^X_\delta$ are both cofinal in $\delta$.  Hence, $D$ does not measure $\mathcal C$.
\end{proof}

The following is proved in joint work of the first author with John Krueger.

\begin{theorem} (\cite{Aspero-Krueger}) $\M_{\al_1}$ can be forced over any model of $\ZFC$ and follows from $\BPFA$. \end{theorem}

Another natural way to strengthen $\Measuring$ is to allow, in the sequence to be measured, not just closed sets, but also sets of higher complexity (from a descriptive set-theoretic point of view). The version of $\Measuring$ where one considers sequences $\vec X=(X_\d\,:\,\d\in\Lim(\o_1))$, with each $X_\d$ an open subset of $\d$ in the order topology, is of course equivalent to $\Measuring$. A natural next step would therefore be to consider sequences in which each $X_\d$ is a countable union of closed sets. This is of course the same as allowing each $X_\d$ to be an arbitrary subset of $\d$. Let us call the corresponding statement $\Measuring^\ast$:

\begin{definition} $\Measuring^\ast$ holds if and only if for every sequence $\vec X=(X_\d\,:\,\d\in\Lim(\o_1))$, if $X_\d\sub\d$ for all $\d$, then there is some club $C\sub\o_1$ such that for every $\d\in C$, a tail of $C\cap\d$ is either contained in or disjoint from $X_\d$.
  \end{definition}

It is easy to see that
$\Measuring^\ast$  is false in $\ZFC$.
As a matter of fact, given a stationary and co-stationary $S\sub\o_1$, there is no club of $\o_1$ measuring $\vec X=(S\cap\d\,:\,\d\in\Lim(\o_1))$. In fact, if $C$ is any club of $\o_1$, then both $C\cap S\cap\d$ and $(C\cap\d)\setminus S$ are cofinal subsets of $\d$ for each $\d$ in the club of limit points in $\o_1$ of both $C\cap S$ and $C\setminus S$.

The status of $\Measuring^\ast$ is more interesting in the absence of the Axiom of Choice. Let $\mtcl C_{\o_1}=\{X\sub\o_1\,:\,C\sub X\mbox{ for some club $C$ of $\o_1$}\}$.

\begin{observation}  ($\ZF$+ $\mtcl C_{\o_1}$ is a normal filter on  $\o_1$)
Suppose $\vec X=(X_\d\,:\,\d\in\Lim(\o_1))$ is such that

\begin{enumerate}
 \item  $X_\d\sub\d$ for each $\d$.
 \item For each club $C\sub\o_1$,

 \begin{enumerate}

 \item there is some $\d\in C$ such that $C\cap X_\d\neq\emptyset$, and

 \item there is some $\d\in C$ such that $(C\cap\d)\setminus X_\d\neq\emptyset$.

\end{enumerate}

\end{enumerate}

Then there is a stationary and co-stationary subset of $\o_1$ definable from $\vec X$.

\end{observation}

\begin{proof}
We have two possible cases. The first case is when for all $\a<\o_1$, either

\begin{itemize}

\item  $W_\a^0=\{\d<\o_1\,:\,\a\notin X_\d\}$ is in $\mtcl C_{\o_1}$, or

\item $W_\a^1=\{\d<\o_1\,:\,\a\in X_\d\}$ is in $\mtcl C_{\o_1}$.

\end{itemize}

For each $\a<\o_1$ let $W_\a$ be $W_\a^\epsilon$ for the unique $\epsilon\in\{0, 1\}$ such that $W_\a^\epsilon\in \mtcl C_{\o_1}$, and let $W^\ast=\Delta_{\a<\o_1} W_\a\in\mtcl C_{\o_1}$. Then $X_{\d_0}=X_{\d_1}\cap\d_0$ for all $\d_0<\d_1$ in $W^\ast$. It then follows, by (2), that $S=\bigcup_{\d\in W^\ast}X_\d$, which of course is definable from $\vec X$, is a stationary and co-stationary subset of $\o_1$. Indeed, suppose $C\sub\o_1$ is a club, and let us fix a club $D\sub W^\ast$. There is then some $\d\in C\cap D$ and some $\a\in C\cap D\cap X_\d$. But then $\a\in S$ since $\d\in W^\ast$ and $\a\in W^\ast\cap X_\d$.   There is also some $\d\in C\cap D$ and some $\a\in C\cap D$ such that $\a\notin X_\d$, which implies that $\a\notin S$ by a symmetrical argument, using the fact that $X_{\d_0}=X_{\d_1}\cap\d_0$ for all $\d_0<\d_1$ in $W^\ast$.

The second possible case is that in which there is some $\a<\o_1$ with the property that both $W^0_\a$ and $W^1_\a$ are stationary subsets of $\o_1$. But now we can let $S$ be $W^0_\a$, where $\a$ is first such that $W^0_\a$ is stationary and co-stationary.
\end{proof}

It is worth comparing the above observation with Solovay's classic result that an $\o_1$-sequence of pairwise disjoint stationary subsets of $\o_1$ is definable from any given ladder system on $\o_1$ (working in the same theory).

\begin{corollary} ($\ZF$+ $\mtcl C_{\o_1}$ is a normal filter on  $\o_1$) The following are equivalent.

\begin{enumerate}

\item $\mtcl C_{\o_1}$ is an ultrafilter on $\o_1$.

\item $\Measuring^\ast$

\item For every sequence $(X_\d\,:\,\d\in\Lim(\o_1))$, if $X_\d\sub\d$ for each $\d$, then there is a club $C\sub\o_1$ such that either

\begin{itemize}

\item $C\cap\d\sub X_\d$ for every $\d\in C$, or
\item $C\cap X_\d=\emptyset$ for every $\d\in C$.
\end{itemize}
\end{enumerate}

\end{corollary}

\begin{proof} (3) trivially implies (2), and by the observation (1) implies (3). Finally, to see that (2) implies (1), note that the argument right after the definition of $\Measuring^\ast$ uses only $\ZF$ together with the regularity of $\o_1$ and the negation of (1).
\end{proof}

In particular, the strong form of $\Measuring^\ast$ given by (3) in the above observation follows from $\ZF$ together with the Axiom of Determinacy.

Much of the notation used in this paper follows the standards set forth in \cite{JECH} and \cite{KUNEN}. Other, less standard, pieces of notation will be introduced as needed. The rest of the paper is structured as follows. In Section \ref{first-construction} we construct a sequence $(\mtcl P_\b\,:\,\b\leq\k)$ of forcing notions. In Section \ref{relevant_facts} we prove the relevant facts about this construction which will show $\mtcl P_\k$ to witness the conclusion of Theorem \ref{mainthm-0-intro}. Subsection \ref{conclusion} contains some remarks on why our construction in Section \ref{first-construction} cannot possibly be adapted to force $\Unif(\vec C)$ for any ladder system $\vec C$ (which, as we already mentioned, is well-known to be incompatible with $\CH$), and on the (closely related) obstacles towards building models of reasonable forcing axioms together with $\CH$ using the present approach.

\section{The main construction}\label{first-construction}

The theorem we will prove in this and the next section, we recall, is the following.

\begin{theorem}\label{mainthm-0} ($\CH$) Let $\k\geq\o_2$ be a regular cardinal such that $2^{{<}\k}=\k$.
Then there is a partial order $\mtcl P\sub H(\o_2)$ with the following properties.
\begin{enumerate}
\item $\mtcl P$ is proper.
\item $\mtcl P$ is $\aleph_2$-Knaster.
\item $\mtcl P$ forces the following statements.
\begin{enumerate}
\item $\Measuring$
\item $\CH$
\item $2^{\aleph_1}=\kappa$
\end{enumerate}
\end{enumerate}
\end{theorem}

In this section we present the construction of a certain sequence $(\mtcl P_\b\,:\,\b\leq\k)$ of forcing notions. In Section \ref{relevant_facts} we will prove that $\mtcl P_\k$ is a forcing $\mtcl P$ witnessing the conclusion of Theorem \ref{mainthm-0}.

We start out by fixing some pieces of notation that will be used in both this and the next section. If $N$ is a set such that $N\cap \omega_1\in \omega_1$, $\delta_N$ denotes this intersection. $\delta_N$ is also called \emph{the height of $N$}.

Given $P\sub H(\kappa)$ and $N\sub H(\kappa)$, we will tend to write $(N, P)$ as short-hand for $(N, P\cap N)$.  Also, if $N_0$ and $N_1$ are $\in$-isomorphic elementary submodels of $H(\k)$, we refer to the unique $\in$-isomorphism $\Psi:(N_0; \in)\to (N_1; \in)$ as $\Psi_{N_0, N_1}$.

We will make use of the following notion of symmetric system from \cite{forcing-conseqs}.

 \begin{definition}\label{hom0}
Let $T\subseteq H(\kappa)$ and let $\mathcal N$ be a finite collection of countable subsets of $H(\kappa)$.
We say that \emph{$\mathcal N$ is a $T$-symmetric system} if and only if the following holds.

\begin{enumerate}

\item For every $N \in\mathcal N$, $(N; \in, T)$ is an elementary substructure of $(H(\kappa); \in, T)$.

\item Given $N_0$ and $N_1$ in $\mathcal N$, if $\delta_{N_0}=\delta_{N_1}$, then there is a unique isomorphism $$\Psi_{N_0, N_1}:(N_0; \in, T)\into (N_1; \in, T)$$

\noindent Furthermore, $\Psi_{N_0, N_1}$ is the identity on $N_0\cap N_1$.

\item For all $N_0$, $N_1$, $M\in \mathcal N$, if $M \in N_0$ and $\delta_{N_0}=\delta_{N_1}$, then $\Psi_{N_0, N_1}(M) \in\mathcal N$.

\item For all $N$ and $M$ in $\mathcal N$, if $\delta_M<\delta_N$, then there is $N'\in\mathcal N$ such that $\delta_{N'}=\delta_N$ and $M\in N'$.
\end{enumerate}

\end{definition}

Taking up a suggestion of Inamdar,  we call condition (4) \emph{the shoulder axiom}.

Strictly speaking, the phrase `$T$-symmetric system' is ambiguous in general since $H(\kappa)$ may not be determined by $T$. However, in all practical cases $(\bigcup T)\cap\Ord = \kappa$, so $T$ does determine $H(\kappa)$ in these cases.

We will talk about \emph{symmetric systems} in some contexts in which $T$ is clear or irrelevant.

The following two amalgamation lemmas  are proved in \cite{forcing-conseqs}.

\begin{lemma}\label{amalg-0}
Let $T\sub H(\k)$ and let $\mtcl N$ be a $T$-symmetric system.
Let $N\in\mtcl N$ and let $\mtcl M\in N$ be a $T$-symmetric system
such that $\mtcl N\cap N\sub\mtcl M$. Let $$\mtcl W(\mtcl N, \mtcl M, N):=\mtcl N\cup\{\Psi_{N, N'}(M)\,:\, M\in\mtcl M,\,N'\in\mtcl N,\,\d_{N'}=\d_N\}$$ Then $\mtcl W(\mtcl N, \mtcl M, N)$ is the $\sub$-minimal $T$-symmetric system $\mtcl W$
such that $\mtcl N\cup\mtcl M\sub\mtcl W$.
\end{lemma}

Given $T\sub H(\k)$ and $\mtcl N_0$ and $\mtcl N_1$, $T$-symmetric systems, let us write $\mtcl N_0\cong_T\mtcl N_1$ if $|\mtcl N_0|=|\mtcl N_1|=n$, for some $n<\o$, and there are enumera\-tions $(N^0_i\,:\, i<n)$ and $(N^1_i\,:\,i<n)$ of $\mtcl N_0$ and $\mtcl N_1$, respectively, for which there is an isomorphism $$\Psi:(\bigcup\mtcl N_0; \in, N^0_i, T)_{i<n}\into (\bigcup\mtcl N_1; \in, N^1_i, T)_{i<n}$$ which is the identity on $(\bigcup\mtcl N_0)\cap (\bigcup \mtcl N_1)$.

\begin{lemma}\label{amalg-1}
Let $T\sub H(\k)$ and let $\mtcl N_0$ and $\mtcl N_1$ be $T$-symmetric systems such that $\mtcl N_0\cong_T\mtcl N_1$. Then $\mtcl N_0\cup\mtcl N_1$ is the $\sub$-minimal $T$-symmetric system $\mtcl W$ such that $\mtcl N_0\cup\mtcl N_1\sub\mtcl W$.
\end{lemma}

We will recursively build a sequence $(\mtcl P_\b\,:\,\b\leq\k)$ of forcing notions, together with a sequence of predicates $(\Phi_\a\,:\,\a<\k)$. Theorem \ref{mainthm-0} will be witnessed by $\mtcl P_\k$. Given $\b<\kappa$ we let $$\mtcl T_\b=\{N\in [H(\k)]^{\al_0}\,:\, (N; \in, \Phi_\b)\elsub (H(\k); \in \, \Phi_\b)\}$$

Let $\Succ(\kappa)$ denote the set of successor ordinals below $\kappa$. To start with, let us fix a function $\Phi:\Succ(\k)\into H(\k)$ with the property that $\{\a\in\Succ(\kappa)\,:\,\Phi(\a)=x\}$ is unbounded in $\k$ for each $x\in H(\k)$ (which exists by $2^{{<}\k}=\kappa$), and let $\Phi_0$ be the satisfaction predicate for the structure $(H(\k); \in, \Phi)$. Also, given any $\b>0$, $\Phi_\b$ will uniformly encode, among other things, the sequences $(\Phi_\a\,:\,\a<\b)$ and $(\text{Sat}(\Phi_\a)\,:\,\a<\b)$, where $\text{Sat}(\Phi_\a)$ denotes the satisfaction predicate for the structure $(H(\k); \in, \Phi_\a)$. 

We will call an ordered pair $(N,  \r)$, where
\begin{itemize}
\item $N$ is a countable elementary submodel of $(H(\k); \in, \Phi_0)$,
\item $\r\in N\cap\k$,
and
\item $N\in\mtcl T_{\a+1}$ for every $\a\in N\cap\r$,
\end{itemize}
a \noindent \emph{model with marker}.\footnote{In the definition of $\mtcl P_\b$, we will assume $\Phi_{\a+1}$ has been defined for all $\a<\b$. While defining $\mtcl P_\b$, we will refer to the notion of model with marker. In that case, the marker $\r$ will be at most $\b$, and hence $\Phi_{\a+1}$---and therefore $\mtcl T_{\a+1}$---will be defined for all $\a\in N\cap\r$.}

If $(N, \r)$ is a model with marker, we will sometimes say that \emph{$\r$ is the marker of $(N, \r)$}.

In our forcing construction, we will use models with markers $(N, \r)$ in a crucial way. The presence of the marker $\r$ will tell us that $N$ is to be seen as `active' for all stages in $N\cap\r$.

Given an unordered pair $$e=\{(N_0, \r_0), (N_1,  \r_1)\}$$ of models with markers, we will call $e$ an \emph{edge} in case

 \begin{enumerate}
 \item $N_0\cong N_1$;
 \item for every $\a\in N_0\cap\r_0$, if $\bar\a=\Psi_{N_0, N_1}(\a)<\r_1$, then
 $\Psi_{N_0, N_1}$ is an isomorphism between $$(N_0; \in, \Phi_{\a+1})$$ and $$(N_1; \in, \Phi_{\bar\a+1}).$$
 \end{enumerate}

 We note that, in the above definition, $(N_0, \r_0)$ and $(N_1, \r_1)$ may or may not be distinct. Hence, an edge may contain two models with markers or may just be the singleton $\{(N, \r)\}$ of a model with marker $(N, \r)$.

Also, we call an ordered pair $\la (N_0, \r_0), (N_1, \r_1)\ra$ a \emph{directed edge} if $\{(N_0, \r_0), (N_1, \r_1)\}$ is an
 edge. If $\mtcl G$ is a set of edges, we say that a directed edge $\la (N_0, \r_0), (N_1, \r_1)\ra$ \emph{comes from $\mtcl G$} if $\{(N_0, \r_0), (N_1, \r_1)\}\in\mtcl G$.

 If $e=\la  (N_0, \r_0), (N_1, \r_1)\ra$ is a directed edge, we write $\Psi_e$ for $\Psi_{N_0, N_1}$.

If $\b<\k$, we say that an edge $\{(N_0, \r_0), (N_1,  \r_1)\}$ is \emph{below $\b$} if $\r_0\leq\b$ and $\r_1\leq\b$.

Given a set $\mtcl G$ of
edges,\footnote{We think of sets of
edges as
graphs, hence the choice of the letter $\mtcl G$ in this context.} we denote $\bigcup\mtcl G$ by $\D(\mtcl G)$; i.e., $\D(\mtcl G)$ is the set of models with markers $(N, \r)$ for which there is some $(N', \r')$ such that $\{(N, \r), (N', \r')\}\in\mtcl G$.

Given a directed edge $e=\langle (N_0, \r_0), (N_1, \r_1)\rangle$ and an edge $e'=\{(N_0', \r_0'), (N_1', \r_1')\}$ such that
\begin{itemize}
\item $e'\in N_0$,
\item $\max\{\r_0', \r_1'\}\leq\r_0$, and
\item $\Psi_{N_0, N_1}(\max\{\r_0', \r_1'\})\leq\r_1$,
\end{itemize}
\noindent we denote $$\{(\Psi_{N_0, N_1}(N_0'), \Psi_{N_0, N_1}(\r_0')),  (\Psi_{N_0, N_1}(N_1'), \Psi_{N_0, N_1}(\r_1'))\}$$ by $\Psi_e(e')$.

\begin{fact} Suppose $e=\langle (N_0, \r_0), (N_1, \r_1)\rangle$ is a directed edge and $e'=\{(N_0', \r_0'), (N_1', \r_1')\}$ is and edge such that
\begin{itemize}
\item $e'\in N_0$,
\item $\max\{\r_0', \r_1'\}\leq\r_0$, and
\item $\Psi_{N_0, N_1}(\max\{\r_0', \r_1'\})\leq\r_1$.
\end{itemize} Then $\Psi_e(e')$ is an
edge. \end{fact}

\begin{proof} For $i\in\{0, 1\}$, let $N_i''=\Psi_{N_0, N_1}(N_i')$. Then, for each $i$, the elementarity of $\Psi_{N_0, N_1}$, together with the fact that $N_0'\cong N_1'$ and $\r_i'\in N_i'$, implies that $N_0''\cong N_1''$ and $\Psi_{N_0, N_1}(\r_i')\in N_i''$. Furthermore, for each $\a\in N_i'\cap\r_i'$, the fact that $\Psi_{N_0, N_1}$ is also an isomorphism between the structures $(N_0; \in, \Phi_{\a+1})$ and $(N_1; \in, \Phi_{\bar\a+1})$, for $\bar\a=\Psi_{N_0, N_1}(\a)$, together with $(N_i'; \in, \Phi_{\a+1})\elsub (N_0; \in, \Phi_{\a+1})$, implies that $$(N_i''; \in, \Phi_{\bar\a+1})\elsub (N_1; \in, \Phi_{\bar\a+1})\elsub (H(\k); \in, \Phi_{\bar\a+1})$$ Hence, $(N_i'', \Psi_{N_0, N_1}(\r_i'))$ is a model with markers. Finally, if $\a$ and $\bar\a$ are as above, with $i=0$, $\b=\Psi_{N_0', N_1'}(\a)$, and $\a^\dag:=\Psi_{N_0'', N_1''}(\bar\a)=\Psi_{N_0, N_1}(\b)<\Psi_{N_0, N_1}(\r_1')$, then letting $\a^*=\max\{\a, \b\}$ and $\a^{**}=\Psi_{N_0, N_1}(\a^*)$ and using the fact that $(N_0'; \in, \Phi_{\a+1})\cong (N_1'; \in, \Phi_{\Psi_{N_0', N_1'}(\a)+1})$ and that $\Psi_{N_0, N_1}$ is also an isomorphism between $(N_0; \in, \Phi_{\a^*+1})$ and $(N_1; \in, \Phi_{\a^{**}+1})$, we get that $(N_0''; \in, \Phi_{\bar\a+1})\cong (N_1''; \in, \Phi_{\a^\dag+1})$. To see this, simply use that $(N_0', \in, \Phi_{\a+1})\elsub (N_0; \in, \Phi_{\a+1})$, $(N_1', \in ,\Phi_{\b+1})\elsub (N_0; \in, \Phi_{\b+1})$ and, if $\a^*>\min\{\a, \b\}$, also  that $\Phi_{\a^*+1}$ codes the satisfaction relation of $(H(\k); \in, \Phi_{\min\{\a, \b\}+1})$.
 \end{proof}

Given a set $\mtcl G$ of edges, we say that $\mtcl G$ is \emph{closed under restrictions} if $\{(N_0, \a_0), (N_1, \a_1)\}\in\mtcl G$ whenever $\{(N_0, \r_0), (N_1, \r_1)\}\in\mtcl G$, $\a_0\in N_0\cap (\r_0+1)$, and $\a_1\in N_1\cap (\r_1+1)$.
Also, we say that $\mtcl G$ is \emph{closed under copying} in case for every directed edge $e=\langle (N_0, \r_0), (N_1, \r_1)\rangle$ coming from $\mtcl G$ and every edge $e'=\{(N'_0, \r'_0), (N'_1, \r'_1)\}\in\mtcl G$, if $e'\in  N_0$, $\max\{\r_0', \r_1'\}\leq\r_0$, and $\Psi_{N_0, N_1}(\max\{\r_0', \r_1'\})\leq\r_1$, then $\Psi_e(e')\in\mtcl G$.

If $\D$ is a set of models with markers and $\b<\k$, we let $$\mtcl N^\D_\b=\{N\,:\, (N, \b)\in\D\}.\footnote{Note that if $\mtcl G$ is a set of edges closed under restrictions and $\D=\D(\mtcl G)$, then  $\mtcl N^\D_0$ is clearly the same thing as $\dom(\D)$.}$$

We say that a set $\mtcl G$ of edges is \emph{sticky} in case for every ordinal $\a$ and for all $N_0$, $N_1\in\mtcl N^{\D(\mtcl G)}_{\a+1}$, if $\d_{N_0}=\d_{N_1}$, then $\{(N_0, \a+1), (N_1, \a+1)\}\in\mtcl G$.\footnote{In particular, if $\mtcl G$ is sticky, then $\{(N, \a+1)\}\in\mtcl G$ for every ordinal $\a$ and every $N\in\mtcl N^{\D(\mtcl G)}_{\a+1}$.}

Given sets $\mtcl G_0$ and $\mtcl G_1$ of edges,
we say that \emph{$\mtcl G_0$ and $\mtcl G_1$ are compatible} in case for all $\a<\k$ and $N_0$, $N_1\in\mtcl N^{\D(\mtcl G_0)}_{\a+1} \cup \mtcl N^{\D(\mtcl G_1)}_{\a+1}$ such that $\d_{N_0}=\d_{N_1}$ we have that $(N_0; \in, \Phi_{\a+1})\cong (N_1; \in, \Phi_{\a+1})$. If this is the case, then
 there is a $\sub$-minimum
sticky
 set $\mtcl G$ of edges including both $\mtcl G_0$ and $\mtcl G_1$ and which is closed under restrictions and closed under copying. We denote this set $\mtcl G$ by $\mtcl G_0\oplus\mtcl G_1$.

If $\mtcl G$ is a set of edges, we denote by $\mtbb M(\mtcl G)$ some canonically chosen structure with universe $\bigcup\dom(\D(\mtcl G))$ coding $\mtcl G$ and $$\langle (\a, \Phi_{\a+1}\cap\bigcup\dom(\D(\mtcl G)))\,:\,\a\in \bigcup\{N\cap\r\,:\,(N, \r)\in\D(\mtcl G)\}\rangle$$ Also, we consider the following form of the isomorphism relation $\cong_T$ for $T$-symmetric systems, for sets of edges: If $\mtcl G_0$ and $\mtcl G_1$ are sets of edges, we write $\mtcl G_0\cong\mtcl G_1$ in case there is an isomorphism $\Psi:\mtbb M(\mtcl G_0)\into\mtbb M(\mtcl G_1)$ which is the identity on $(\bigcup\dom(\D(\mtcl G_0)))\cap (\bigcup\dom(\D(\mtcl G_1)))$.

We will use the following easy extension of Lemma \ref{amalg-1}.

\begin{lemma}\label{amalg-1+}
Let $\mtcl G_0$ and $\mtcl G_1$ be
sticky
sets of edges closed under restrictions and under copying. Suppose $\mtcl G_0\cong \mtcl G_1$. Then $\mtcl G_0\oplus\mtcl G_1$ is the union of $\mtcl G_0\cup \mtcl G_1$
 and the set of unordered pairs $\{(N_0, \a_0+1), (N_1, \a_1+1)\}$ such that $\d_{N_0}=\d_{N_1}$, $\a_0\in N_0$, $\a_1\in N_1$, and for which there is some $\a\geq\a_0$, $\a_1$ such that $N_0\in\mtcl N^{\D(\mtcl G_0)}_{\a+1}$ and $N_1\in\mtcl N^{\D(\mtcl G_1)}_{\a+1}$.\footnote{We note that, in particular, $\mtcl G_0$ and $\mtcl G_1$ are compatible, and so $\mtcl G_0\oplus\mtcl G_1$ exists.}
Hence, if, in addition, $\mtcl N^{\D(\mtcl G_0)}_0$ and $\mtcl N^{\D(\mtcl G_1)}_0$ are $\Phi_0$-symmetric systems and $\mtcl N^{\D(\mtcl G_0)}_{\a+1}$ and $\mtcl N^{\D(\mtcl G_1)}_{\a+1}$ are $\Phi_{\a+1}$-symmetric systems for each $\a<\k$, then $\mtcl N^{\D(\mtcl G_0\oplus\mtcl G_1)}_0$ is a $\Phi_0$-symmetric system and $\mtcl N^{\D(\mtcl G_0\oplus\mtcl G_1)}_{\a+1}$ is a $\Phi_{\a+1}$-symmetric system for each $\a<\k$.
\end{lemma}

If $\mtcl G$ is a set of edges and $\a<\k$, we let $$\mtcl G\av_\a=\{\{(N_0, \r_0), (N_1, \r_1)\}\in\mtcl G\,:\, \r_0, \r_1\leq \a\}$$

We will need the following easy lemma.

\begin{lemma}\label{amalg-restr-sym}
Suppose $\mtcl G$ is a
sticky
set of edges closed under restrictions and under copying. Suppose $\mtcl N^{\D(\mtcl G)}_0$ is a $\Phi_0$-symmetric system and $\mtcl N^{\D(\mtcl G)}_{\a+1}$ is a $\Phi_{\a+1}$-symmetric system for each $\a<\k$. Let $\a_0<\k$. Then the following holds.
\begin{enumerate}
\item $\mtcl G\av_{\a_0}$ is a sticky set of edges closed under restrictions and under copying.
\item $\mtcl N^{\D(\mtcl G\av_{\a_0})}_\a=\mtcl N^{\D(\mtcl G)}_\a$ for every $\a\leq\a_0$. In particular, $\mtcl N^{\D(\mtcl G\av_{\a_0})}_0$ is a $\Phi_0$-symmetric system and for each $\a<\k$,  $\mtcl N^{\D(\mtcl G\av_{\a_0})}_{\a+1}$ is a $\Phi_{\a+1}$-symmetric system.
\end{enumerate}
\end{lemma}

Given functions $f_0,\ldots, f_n$, for some $n<\o$, we let $$f_n \circ \ldots\circ f_0$$ be $f_0$ if $n=0$; if $n>0$, we let this expression denote the function $f$ with domain the set of $x$ such that for every $i<n$, $x\in\dom(f_i\circ\ldots\circ f_0)$ and $(f_i\circ\ldots\circ f_0)(x)\in\dom(f_{i+1})$, and such that for every $x\in \dom(f)$, $f(x)=f_n((f_{n-1}\circ\ldots\circ f_0)(x))$.

If $\vec{\mtcl E}=(\langle (N^i_0, \r^i_0), (N^i_1, \r^i_1)\rangle\,:\,\,i< n)$, for some $n<\o$, is a sequence of pairs of models with markers such that $N^i_0\cong N^i_1$ for all $i<n$, we denote $\Psi_{N^{n-1}_0, N^{n-1}_1}\circ\ldots\circ\Psi_{N^0_0, N^0_1}$ by $\Psi_{\vec{\mtcl E}}$. We also let $\d_{\vec{\mtcl E}}=\{\d_{N^i_0}\,:\,i<n\}$.

If $\mtcl G$ is a set of edges and $a\in H(\k)$, we call $\la a, \vec{\mtcl E}\ra$ a \emph{$\mtcl G$-thread} if $\vec{\mtcl E}$ is a finite sequence of directed edges coming from $\mtcl G$ and $a\in\dom(\Psi_{\vec{\mtcl E}})$.

Given a set $\mtcl G$ of edges and an ordinal $\a<\k$, we say that $$\la \a, (\la (N^i_0, \r^i_0), (N^i_1, \r^i_1) \ra\,:\, i\leq n)\ra$$  is a \emph{connected $\mtcl G$-thread} in case the following holds.

\begin{enumerate}
\item $\la \a, (\la (N^i_0, \r^i_0), (N^i_1, \r^i_1) \ra\,:\, i\leq n)\ra$ is a $\mtcl G$-thread.
\item $\a\in N^0_0\cap (\r^0_0+1)$ and $\Psi_{N^0_0, N^0_1}(\a)<\r^0_1+1$.
\item If $n>0$, then $\la (\Psi_{N^0_0, N^0_1}(\a), (\la (N^i_0, \r^i_0), (N^i_1, \r^i_1) \ra\,:\, 0<i\leq n)\ra$ is a connected $\mtcl G$-thread.
\end{enumerate}

If $\mtcl G$ is a set of edges and $(\d, \a)$, $(\d, \bar\a)\in\o_1\times\k$, we say that \emph{$(\d, \bar\a)$ is $\mtcl G$-accessible from $(\d, \a)$} if
\begin{itemize}
\item $\bar\a=\a$ or
\item there is a connected $\mtcl G$-thread $\la \a, \vec{\mtcl E}\ra$ such that $\bar\a=\Psi_{\vec{\mtcl E}}(\a)$ and $\d\leq\min(\d_{\vec{\mtcl E}})$.

\end{itemize}

In the proof of Lemma \ref{amalg-0+}, if $$\vec{\mtcl E}=(\langle (N^i_0, \r^i_0), (N^i_1, \r^i_1)\rangle\,:\,\,i< n)$$ is a sequence of ordered edges, we will denote the sequence $$(\langle (N^{n-1-i}_1, \r^{n-1-i}_1), (N^{n-1-i}_0, \r^{n-1-i}_0)\rangle\,:\,\,i< n)$$ by $(\vec{\mtcl E})^{-1}$.

We will need the following counterpart of Lemma \ref{amalg-0} for sets of edges.

\begin{lemma}\label{amalg-0+}
Let $\b<\k$. Let $\mtcl G_0$ be a sticky set of edges below $\b$ closed under restrictions and under copying and such that $\mtcl N^{\D(\mtcl G_0)}_0$ is a $\Phi_0$--symmetric system and $\mtcl N^{\D(\mtcl G_0)}_{\a+1}$ is a $\Phi_{\a+1}$--symmetric system for each $\a<\k$. Let $N\in\mtcl N^{\D(G_0)}_\b$. Suppose $\mtcl G_1\in N$ is a sticky set of edges below $\b$ closed under restrictions and under copying and such that $\mtcl N^{\D(\mtcl G_1)}_0$ is a $\Phi_0$--symmetric system and $\mtcl N^{\D(\mtcl G_1)}_{\a+1}$ is a $\Phi_{\a+1}$--symmetric system for each $\a<\k$. Suppose $\mtcl G_0\cap N\sub\mtcl G_1$. Finally, suppose that for every $Q\in\dom(\D(\mtcl G_0))\cap N$,
$\mtcl G_1\cap Q=\mtcl G_0\cap Q$. 
Let $\mtcl G^*$ be the union of the following sets.
\begin{enumerate}
\item $\mtcl G_0$
\item The set $\mtcl G_2$ consisting of unordered pairs of the form $$\{(\Psi_{\vec{\mtcl E}}(N_0), \Psi_{\vec{\mtcl E}}(\r_0)), (\Psi_{\vec{\mtcl E}}(N_1), \Psi_{\vec{\mtcl E}}(\r_1))\},$$ where $\{(N_0, \r_0), (N_1, \r_1)\}\in\mtcl G_1$, $\la\{N_0, N_1\}, \vec{\mtcl E}\ra$ is a $\mtcl G_0$-thread with
$\min(\d_{\vec{\mtcl E}})=\d_N$,
 and $\la\r_0, \vec{\mtcl E}\ra$ and $\la\r_1, \vec{\mtcl E}\ra$ are connected $\mtcl G_0$-threads.
\item The set $\mtcl G_3$ consisting of unordered pairs of the form $$\{(M_0, \a_0), (M_1, \a_1)\}$$ such that $\d_{M_0}=\d_{M_1}$ and for which there is some $\a<\b$ such that $\{(M_0, \a+1)\} \in \mtcl G_2$, $\{(M_1, \a+1)\} \in \mtcl G_2$,  $\a_0\in M_0\cap (\a+2)$, and $\a_1\in M_1\cap (\a+2)$.
\end{enumerate}
Then $\mtcl G^*$ is a sticky set of edges closed under restrictions and under copying, $\mtcl N^{\D(\mtcl G^*)}_0$ is a $\Phi_0$--symmetric system, and $\mtcl N^{\D(\mtcl G^*)}_{\a+1}$ is a $\Phi_{\a+1}$--symmetric system for each $\a<\k$.
\end{lemma}

\begin{proof}
It is immediate to check that, by our construction, $\mtcl G^*$ is closed under restrictions. Also, it is clear that $\mtcl N^{\Delta(\mtcl G^*)}_0=\mtcl N^{\Delta(\mtcl H)}_0$, where $$\mtcl H=\mtcl G_0\cup \{\{(\Psi_{N, N'}(M), 0)\}\,:\, M\in\mtcl N^{\D(\mtcl G_1)}_0, N'\in\mtcl N^{\D(\mtcl G_0)}_0,\,\d_{N'}=\d_N\}$$ Hence, by Lemma \ref{amalg-0}, $\mtcl N^{\mtcl G^*}_0$ is a $\Phi_0$-symmetric system. We will now prove, for every $\a<\b$, that $\mtcl N^{\D(\mtcl G^*)}_{\a+1}$ is a $\Phi_{\a+1}$-symmetric system. The point that needs the most work is the verification of the shoulder axiom for $\mtcl N^{\D(\mtcl G^*)}_{\a+1}$, which we will go through next.



For this, given $M_0^*$, $M_1^*\in\mtcl N^{\mtcl G^*}_{\a+1}$ such that $\d_{M_0^*}<\d_{M_1^*}$,  it is enough to show that there is some $M_1^{**}\in\mtcl N^{\D(\mtcl G^*)}_{\a+1}$ such that $\d_{M_1^{**}}=\d_{M_1^*}$ and $M_0^*\in M_1^{**}$.  If $\d_{M_0^*}\geq\d_N$, then $M_0^*$ and $M_1^*$ are both in $\dom(\D(\mtcl G_0))$ and so we are done by the shoulder axiom for $\mtcl N^{\D(\mtcl G_0)}_{\a+1}$. Hence, we will assume in what follows that $\d_{M_0^*}<\d_N$. If $M_0^*\in\mtcl N^{\D(\mtcl G_0)}_{\a+1}$, then we may of course assume that $M_1^*\notin \mtcl N^{\D(\mtcl G_0)}_{\a+1}$. It then follows, by the definition of $\mtcl G_2$, together with the stickiness of $\mtcl G_0$ and the shoulder axiom for $\mtcl N^{\D(\mtcl G_0)}_{\a+1}$, that there is a sequence $\vec{\mtcl E}$ such that $\la M_0^*, \vec{\mtcl E}\ra$ is a $\mtcl G_0$-thread with $\min(\d_{\vec{\mtcl E}})=\d_N$, $\la \a+1,\vec{\mtcl E}\ra$ is a connected $\mtcl G_0$-thread,
and $\Psi_{\vec{\mtcl E}}(M^*_0)\in N$. Then $M_0:=\Psi_{\vec{\mtcl E}}(M_0^*)\in\dom(\Delta(\mtcl G_0))\cap N$, and therefore $M_0\in\dom(\Delta(\mtcl G_1))$.

For $i=0$, $1$, let us fix $\a_i<\b$, $M_i\in\mtcl N^{\D(\mtcl G_1)}_{\a_i+1}$, and $\vec{\mtcl E}_i$ be such that $\la (M_i, \a_i+1), \vec{\mtcl E}_i\ra$ is a $\mtcl G_0$-thread,
$\min(\d_{\vec{\mtcl E}_i})=\d_N$,  and $\la \a_i+1, \vec{\mtcl E}_i\ra$ is a connected $\mtcl G_0$-thread. Suppose $\a=\Psi_{\vec{\mtcl E}_0}(\a_0)=\Psi_{\vec{\mtcl E}_1}(\a_1)$ and $\d_{M_0}<\d_{M_1}$. By the analysis in the previous paragraph, in order to show the shoulder axiom for $\mtcl N^{\D(\mtcl G^*)}_{\a+1}$ it will suffice to prove that there is some $M_1'\in\mtcl N^{\D(\mtcl G^*)}_{\a+1}$ such that $\d_{M_1'}=\d_{M_1}$ and $\Psi_{\vec{\mtcl E}_0}(M_0)\in M_1'$. By, if necessary, appending suitable ordered edges from $\mtcl G_0$ at the right places using stickiness of $\mtcl G_0$ and the shoulder axiom for $\mtcl N^{\Delta(\mtcl G_0)}_{\g+1}$ for appropriate $\g$---these places could be the beginning or the end of $\vec{\mtcl E}_0$, the beginning or the end of $\vec{\mtcl E}_1$, or somewhere inside $\vec{\mtcl E}_0$ or $\vec{\mtcl E}_1$---
we obtain $\vec{\mtcl E}_0'$ and $\vec{\mtcl E}_1'$ such that $$\Psi_{\vec{\mtcl E}'_1}^{-1}\circ\Psi_{\vec{\mtcl E}'_0}:(N; \in)\into (N; \in)$$ is an isomorphism.  But then $\Psi_{\vec{\mtcl E}'_1}^{-1}\circ\Psi_{\vec{\mtcl E}'_0}\restr N$ is of course the identity on $N$, which implies that $\a_0=\a_1$ since $\Psi_{\vec{\mtcl E}'_1}^{-1}\circ\Psi_{\vec{\mtcl E}'_0}(\a_0)=\a_1$ from the way we have constructed $\vec{\mtcl E}_0'$ and $\vec{\mtcl E}_1'$ from $\vec{\mtcl E}_0$ and $\vec{\mtcl E}_1$, respectively. Now, by the shoulder axiom for $\mtcl N^{\D(\mtcl G_1)}_{\a_0+1}$, we can find $M_1^\dag\in\mtcl N^{\D(\mtcl G_1)}_{\a_0+1}$ such that $\d_{M_1^\dag}=\d_{M_1}$ and $M_0\in M_1^\dag$, and $M_1^{**}:=\Psi_{\vec{\mtcl E}_0}(M_1^\dag)$ is then a model in $\mtcl N^{\mtcl G^*}_{\a+1}$ as desired.

Similarly, by an argument as in the above proof of the shoulder axiom, we can see that if $M_0$, $M_1\in\mtcl N^{\D(\mtcl G^*)}_{\a+1}$ are such that $\d_{M_0}=\d_{M_1}$, then $(M_0; \in, \Phi_{\a+1})\cong (M_1; \in, \Phi_{\a+1})$. More specifically, and as in the proof of the shoulder axiom, we may assume that we are in the case in which for each $i\in\{0, 1\}$ there are $\a_i<\b$, $M^-_i\in\mtcl N^{\D(\mtcl G_1)}_{\a_i+1}$, and $\vec{\mtcl E}_i$ such that $\la (M^-_i, \a_i+1), \vec{\mtcl E}_i\ra$ is a $\mtcl G_0$-thread,
$\min(\d_{\vec{\mtcl E}_i})=\d_N$, $\la \a_i+1, \vec{\mtcl E}_i\ra$ is a connected $\mtcl G_0$-thread, and $\Psi_{\vec{\mtcl E}_i}(M_i^-)=M_i$. To see that  $(M_0; \in, \Phi_{\a+1})\cong (M_1; \in, \Phi_{\a+1})$, we notice that $\a_0=\a_1$ as in the previous argument and therefore $(M_0^-; \in, \Phi_{\a_0+1})\cong (M_1^-; \in, \Phi_{\a_1+1})$. Also, by the same construction as in the argument in the proof of the shoulder axiom, we may obtain $\vec{\mtcl E}_0'=(\langle (N^{i, 0}_0, \r^{i, 0}_0), (N^{i, 0}_1, \r^{i, 0}_1)\rangle\,:\,\,i\leq n_0)$ and $\vec{\mtcl E}_1'=(\langle (N^{i, 1}_0, \r^{i, 1}_0), (N^{i, 1}_1, \r^{i, 1}_1)\rangle\,:\,\,i\leq n_1)$ from $\vec{\mtcl E}_0$ and $\vec{\mtcl E}_1$, so that $\dom(\vec{\mtcl E}'_0)=\dom(\vec{\mtcl E}'_1)=N$, $\Psi_{\vec{\mtcl E}_0'}(M_0^-)=M_0$, and  $\Psi_{\vec{\mtcl E}_1'}(M_0^-)=M_1$. 
But then the desired conclusion holds since $$\Psi_{\vec{\mtcl E}'_0}:(N; \in, \Phi_{\a_0+1})\into (N^{n_0, 0}_1; \in, \Phi_{\a+1})$$ and $$\Psi_{\vec{\mtcl E}'_1}:(N; \in, \Phi_{\a_0+1})\into (N^{n_1, 1}_1; \in, \Phi_{\a+1})$$ are isomorphisms. The proof that $(\Psi_{M_0, M_1}(M), \a+1)\in\D(\mtcl G^*)$ whenever $M_0$, $M_1$ are as above and $M\in\mtcl N^{\D(\mtcl G^*)}_{\a+1}\cap M_0$, which concludes the proof that $\mtcl N^{\D(\mtcl G^*)}_{\a+1}$ is a $\Phi_{\a+1}$-symmetric system, is contained in the argument in the next paragraph.

We now show that $\mtcl G^*$ is closed under copying. For this, suppose $e=\{(M_0, \r_0), (M_1, \r_1)\}\in\mtcl G^*$ and $e'=\{(M_0', \r_0'), (M_1', \r_1')\}\in\mtcl G^*\cap M_0$ are such that  $\max\{\r_0', \r_1'\}\leq\r_0$ and $\Psi_{N_0, N_1}(\max\{\r_0', \r_1'\})\leq\r_1$, and let us prove that $\Psi_{M_0, M_1}(e')\in\mtcl G^*$. The case when $\d_{M_0}\geq\d_N$ follows from the construction of $\mtcl G_2$ -- in this case of course $M_0$, $M_1\in\mtcl N^{\D(\mtcl G_0)}_{\a+1}$. Now suppose $\d_{M_0}<\d_N$. If $e\in\mtcl G_2$, then the conclusion follows from the construction of $\mtcl G_2$ and the hypothesis that $Q\cap\mtcl G_1=Q\cap\mtcl G_0$ for every $Q\in\dom(\D(\mtcl G_0))\cap N$. In order to finish this proof it thus remains to consider the case in which $e\in \mtcl G_3$.  We then have that there is $\a+1\geq\r_0$, $\r_1$ such that the edges $\{(M_0, \a+1)\}$ and $\{(M_1, \a+1)\}$ are both in $\mtcl G_2$. Hence there are $\a^*<\b$ and $\{(M_0^*, \a^*+1)\}$, $\{(M_1^*, \a^*+1)\}\in\mtcl G_1$ such that $M_0=\Psi_{\vec{\mtcl E}_0}(M_0^*)$ and $M_1=\Psi_{\vec{\mtcl E}_1}(M_1^*)$ for suitable $\vec{\mtcl E}_0$ and $\vec{\mtcl E}_1$ as in the definition of $\mtcl G_2$ such that $\Psi_{\vec{\mtcl E}_0}(\a^*)=\Psi_{\vec{\mtcl E}_1}(\a^*)=\a$. Since then $\{(M_0^*, \a^*+1), (M_1^*, \a^*+1)\}\in\mtcl G_1$ by stickiness of $\mtcl G_1$ and $\Psi_{\vec{\mtcl E_0}}^{-1}(e')\in\mtcl G_1\cap M_0^*$,  $e^*:=\Psi_{M_0^*, M_1^*}(\Psi_{\vec{\mtcl E_0}}^{-1}(e'))\in\mtcl G_1$. This finishes the proof in this case since then $\Psi_{M_0, M_1}(e')=\Psi_{\vec{\mtcl E}_1}(e^*)\in\mtcl G_2\sub\mtcl G^*$.  


Finally, we note that stickiness of $\mtcl G^*$ holds at $\a+1$ (i.e., the unordered pair $\{(M_0, \a+1), (M_1, \a+1)\}\in\mtcl G^*$ for all $M_0$, $M_1\in\mtcl N^{\Delta(\mtcl G^*)}_{\a+1}$ such that $\d_{M_0}=\d_{M_1}$) since, by the definition of $\mtcl G_2$,  we can assume that $\{(M_0, \a+1), (M_1, \a+1)\}\notin \mtcl G_0$, $\d_{M_0}=\d_{M_1}<\d_N$, and hence $$\{(M_0, \a+1), (M_1, \a+1)\}\in\mtcl G_3.$$

\end{proof}

\begin{remark} The set $\mtcl G^*$ in the proof of Lemma \ref{amalg-0+} is precisely $\mtcl G_0\oplus\mtcl G_1$. \end{remark}

\begin{remark}\label{inamdar}
The main reason for requiring our sets of edges $\mtcl G$ to be sticky, rather than simply asking that $\mtcl N^{\D(\mtcl G)}_{\a+1}$ be a $\Phi_{\a+1}$-symmetric system for each $\a$, it to secure the above amalgamation lemma. As observed by Inamdar, this lemma does not hold if we do not require stickiness.
\end{remark}

We will call a function $F$ \emph{pertinent} if $\dom(F)\in [\Succ(\kappa)]^{{<}\o}$ and for every $\a\in \dom(F)$,
 $F(\a)=(b_\a, d_\a)$, where

\begin{itemize}
 \item $b_{\alpha}\in [\Lim(\o_1)\times\o_1]^{{<}\o}$ is a regressive function (i.e., $b_\alpha(\d)<\d$ for each $\d\in\dom(b_\a)$);
\item $d_\alpha\in [\o_1\times H(\k)]^{{<}\o}$.
\end{itemize}

In the above situation, we will often refer to $b_\a$ and $d_\a$ as, respectively, $b^F_\a$, and $d^F_\alpha$. Also, if $\a\notin \dom(F)$, $b^F_\a$ and $d^F_\a$ are both defined to be the empty set.

Given an ordered pair $q=(F, \mtcl G)$, where $F$ is a function and $\mtcl G$ is a set of edges,
we will denote $F$ and $\mtcl G$ by, respectively, $F_q$ and $\mtcl G_q$. Given $\a\in \dom(F_q)$, we will denote $b^{F_q}_\a$ and $d^{F_q}_\a$ by, respectively, $b^q_\a$ and $d^q_\a$.

If $q=(F_q,  \mtcl G_q)$, where $F_q$ and $\mtcl G_q$ are as above, and $\b<\k$, we let
$\mtcl N^q_\beta$ stand for $\mtcl N^{\D(\mtcl G_q)}_\b$. If $G$ is a set of ordered pairs as above, we denote by $\mtcl N^G_\b$ the set $\bigcup\{\mtcl N^q_\b\,:\,q\in G\}$.

Given  $q=(F_q, \mtcl G_q)$, where $F_q$ and $\mtcl G_q$ are as above, and given $N\sub H(\k)$, we denote by $q\restr N$ the ordered pair $(F_q\restr\restr N, \mtcl G_q\cap N)$, where $F_q\restr\restr N$ is the function with domain $\dom(F_q)\cap N$ such that $$(F_q\restr\restr N)(\a)=(b^q_\a\cap N, d^q_\a\cap N)$$ for each $\a\in \dom(F)\cap N$.

Also, given $q=(F_q, \mtcl G_q)$ as above, $\d<\o_1$, and $\a<\k$, we denote by $\Xi^{q, \a}_\d$ the set of ordinals $\bar\a$ such that $(\d, \bar\a)$ is $\mtcl G_q$-accessible from $(\d, \a)$, $\bar\a\in\dom(F_q)$, and $\d\in \dom(b^q_{\bar\a})$.

We will now define our sequence $(\mathcal P_\beta\,:\,\beta\leq\k)$ and  $(\Phi_\b\,:\,\b<\k)$. As we said before, Theorem \ref{mainthm-0} will be witnessed by $\mtcl P_{\k}$. We already defined $\Phi_0$.

Given $\a\leq\k$,
 $\dot G_\alpha$ will be the canonical $\mathcal P_\alpha$-name for the generic filter added by $\mtcl P_\a$.
We will denote the forcing relation for $\mtcl P_\a$ by $\Vdash_\a$, and the extension relation for $\mtcl P_\a$ by $\leq_\a$.

Given any $\a<\k$, and assuming $\mtcl P_\a$ has been defined, we let $\dot C^{\a}$ be some canonically chosen (using $\Phi$) $\mathcal P_{\a}$-name for a club-sequence on $\omega_1^V$ for which the following holds.

\begin{itemize}

\item If $\Phi(\a)$ is a $\mathcal P_{\a}$-name for a club-sequence on $\omega_1$, then $\dot C^{\a} = \Phi(\a)$.

\item If $\Phi(\a)$ is not a $\mathcal P_{\a}$-name for a club-sequence on $\omega_1$, then $\dot C^{\a}$ is a $\mtcl P_\a$-name for $\vec C$, where $\vec C\in V$ is some fixed club-sequence on $\o_1$.

\end{itemize}

Given $\d\in\Lim(\o_1)$, we let $\dot C^\a_\d$ be a $\mtcl P_\a$-name for $\dot C^\a(\d)$ (where $\dot C^\a(\d)$ of course refers to the $\d$-th member of $\dot C^\a$).

We are finally in a position to define our construction. Let $\beta <\k$, and suppose $\mathcal P_\alpha$, $\Phi_\a$ and $\Phi_{\a+1}$
have been defined for each $\a<\b$. Suppose, in addition, that for all $\bar\a<\a<\b$, every $\mtcl P_{\bar\a}$-name is also a $\mtcl P_\a$-name. We aim to define $\mtcl P_\b$ and $\Phi_{\b+1}$, and also $\Phi_\b$ if $\b<\k$ is a nonzero limit ordinal.

An ordered pair $q=(F_q, \mtcl G_q)$ is a $\mathcal P_\beta$-condition if and only if it has the following properties.

\begin{enumerate}

\item $\mtcl G_q$ is a sticky set of edges below $\b$ closed under restrictions and under copying, and such that:

\begin{enumerate}
\item $\mtcl N^{\D(\mtcl G_q)}_0$ is a $\Phi_0$-symmetric system;
\item for every $\a<\b$, $\mtcl N^{\D(\mtcl G_q)}_{\a+1}$ is a $\Phi_{\a+1}$-symmetric system.
\end{enumerate}

\item $F_q$ is a pertinent function with $\dom(F_q)\sub \b$.

\item For every $\alpha<\beta$, \emph{the restriction of $q$ to $\alpha$}, $q\arrowvert_{\alpha}$, is a condition in $\mathcal P_\alpha$, where $$q\av_{\a}:=(F_q\restr \alpha, \mtcl G_q\av_\alpha)$$

\item If $\a\in \dom(F_q)$, then $F_q(\a)=(b^q_\a, d^q_\a)$
has the following properties.

\begin{enumerate}
\item For every $\d\in \dom(b^q_\a)$ there is some $N\in\mtcl N^q_{\a+1}$ such that $\d=\d_N$.

\item For every  $N\in\mtcl N^q_{\a+1}$ and $\d\in \dom(b^q_\a)$, if $b^q_\a(\d)<\d_N<\d$ and $\b=\a+1$, then $q\av_\a\Vdash_\a\d_N\notin\dot C^\a_\d$.

\item For every $N\in\mtcl N^q_{\a+1}$, $(\d, a)\in d^q_\a\cap N$ and $N'\in\mtcl N^q_{\a+1}$, if $\d_{N'}=\d_N$, then $(\d, \Psi_{N, N'}(a))\in d^q_\a$.

\item For every $(\d, a)\in d^q_\a$ and $N\in\mtcl N^q_{\a+1}$, if $\d<\d_N$, then there is some $N'\in\mtcl N^q_{\a+1}$ such that $\d_{N'}=\d_N$ and $a\in N'$.

\end{enumerate}

\item Suppose $\b=\a+1$. For every $N\in\mtcl N^q_{\a+1}$, if $\Xi^{q, \a}_{\d_N}\neq\emptyset$, then $q\av_\a$ forces that for every $a\in N$ there is some $M\in\mtcl N^{\dot G_\a}_\a\cap\mtcl T_{\a+1}\cap N$ such that

\begin{enumerate}
\item $a\in M$ and
\item $\d_M\notin\bigcup\{\dot C^{\bar\a}_{\d_N}\,:\,\bar\a\in\Xi^{q, \a}_{\d_N}\}$.\footnote{It is worth noting that clauses (4)(b) and (5) only apply when $\beta=\alpha+1$. Also, notice that item (b) in (5) makes sense since, in the situation of this clause, every $\mtcl P_{\bar\a}$-name is itself a $\mtcl P_\a$-name by our working hypothesis.}
\end{enumerate}

 \item Suppose $\{(N_0,  \r_0), (N_1, \r_1)\}\in\mtcl G_q$, $\a\in \dom(F_q)\cap N_0\cap\r_0$, and $\bar\a=\Psi_{N_0, N_1}(\alpha)<\r_1$.
Then:

\begin{enumerate}
\item $\bar\a\in \dom(F_q)$;
\item $b^q_\a\cap N_0= b^q_{\bar\a}\cap N_1$;
\item $\Psi_{N_0, N_1}``d^q_\a= d^q_{\bar\a}\cap N_1$.
\end{enumerate}


\item The following holds for every $\a<\b$ and every $N\in\mtcl N^q_{\a+1}$.

\begin{enumerate}
\item For all $Q\in\mtcl N^q_{\a+1}\cap N$, and $(\d_0, \d_1)\in b^q_\a$, if $\d_1<\d_Q<\d_0$ and $\d_0<\d_N$, then there is some $p\in \mtcl P_\a\cap N$ such that $q\av_\a\leq_\a p$ and $p\Vdash_\a \d_Q\notin\dot C^\a_{\d_0}$.
\item For every $Q\in\mtcl N^q_{\a+1}\cap N$, if $\Xi^{(q\restr N)|_{\a+1}, \a}_{\d_Q}\neq\emptyset$, then there is some $p\in \mtcl P_\a\cap N$ such that $q\av_\a\leq_\a p$ and such that $p$ forces that for every $a\in Q$ there is some  $M\in\mtcl N^{\dot G_\a}_\a\cap\mtcl T_{\a+1}\cap Q$ with $a\in M$ and $\d_M\notin\bigcup\{\dot C^{\bar\a}_{\d_Q}\,:\,\bar\a\in\Xi^{(q\restr N)|_{\a+1}, \a}_{\d_Q}\}$.\footnote{Just to be clear, $\Xi^{(q\restr N)|_{\a+1}, \a}_{\d_Q}$ is of course the set of ordinals $\bar\a$ such that $(\d_Q, \bar\a)$ is $(\mtcl G_q)|_{\a+1}\cap N$-accessible from $(\d_Q, \a)$, $\bar\a\in\dom(F_q)\cap N$, and $\d_Q\in \dom(b^q_{\bar\a})$.}
\end{enumerate}

\end{enumerate}

Given $\mathcal P_\beta$-conditions $q_i$, for $i =0$, $1$, $q_1\leq_\b q_0$ if and only if the following holds.

\begin{enumerate}
\item $\dom(F_{q_0})\sub \dom(F_{q_1})$ and for every $\alpha\in \dom(F_{q_0})$,
\begin{enumerate}
\item $b^{q_0}_\alpha\sub b^{q_1}_\alpha$ and
\item $d^{q_0}_\alpha\sub d^{q_1}_\alpha$.
\end{enumerate}
\item $\mtcl G_{q_0}\sub\mtcl G_{q_1}$
\item For every $\{(N_0, \r_0), (N_1, \r_1)\}\in\mtcl G_{q_0}$ and $\a\in N_0\cap (\r_0+1)$, the following holds.
\begin{enumerate}
\item If $\Psi_{N_0, N_1}(\a)>\b$, then $\mtcl N^{q_1}_\a\cap N_0=\mtcl N^{q_0}_\a\cap N_0$.
\item If $\a\in \dom(F_{q_1})\cap\r_0$ and $\Psi_{N_0, N_1}(\a)\geq\b$, then:
\begin{enumerate}
\item if $b^{q_1}_\a\cap N_0\neq\emptyset$, then $\a\in\dom(F_{q_0})$ and $b^{q_1}_\a\cap N_0=b^{q_0}_\a\cap N_0$;
\item if $d^{q_1}_\a\cap N_0\neq\emptyset$, then $\a\in\dom(F_{q_0})$ and $d^{q_1}_\a\cap N_0=d^{q_0}_\a\cap N_0$.
\end{enumerate}
\end{enumerate}
\end{enumerate}

We will refer to clause (7) of the definition of $\mtcl P_\b$ holding for $q$ by saying that $q$ is \emph{$N$-saturated below $\b$}.

\begin{fact} $\leq_\b$ is a transitive relation. \end{fact}

\begin{proof}
Let $q_0$, $q_1$, $q_2\in\mtcl P_\b$ and suppose $q_1\leq_\b q_0$ and $q_2\leq_\b q_1$. In order to show that $q_2\leq_\b q_0$, it suffices to verify (3) as all other clauses are trivial. For this, let $\{(N_0, \r_0), (N_1, \r_1)\}\in\mtcl G_{q_0}$, $\a\in N_0\cap (\r_0+1)$ and $\bar\a=\Psi_{N_0, N_1}(\a)$, and let us assume that $\bar\a>\b$. We will prove that $\mtcl N^{q_2}_\a\cap N_0=\mtcl N^{q_0}_\a\cap N_0$. (The argument taking care of (3)(b) is the same.)

Since $\mtcl G_{q_0}\sub\mtcl G_{q_1}\sub\mtcl G_{q_2}$, by (3)(a) in the definition of $q_2\leq_\b q_1$ we have that $\mtcl N^{q_2}_\a\cap N_0=\mtcl N^{q_1}_\a\cap N_0$. Since $\mtcl N^{q_1}_\a\cap N_0 = \mtcl N^{q_0}_\a\cap N_0$ by (3)(a) in the definition of $q_1\leq_\b q_0$, we have that  $\mtcl N^{q_1}_\a\cap N_0=\mtcl N^{q_0}_\a\cap N_0$. Putting these two equalities together it follows that $\mtcl N^{q_2}_\a\cap N_0=\mtcl N^{q_0}_\a\cap N_0$.
\end{proof}

We still need to define $\Phi_{\b+1}$, and $\Phi_\b$ if $\b<\k$ is a nonzero limit ordinal.

Let $\Vdash_\b^*$ denote the restriction of the forcing relation $\Vdash_\b$ for $\mtcl P_\b$ to formulas involving only names in $H(\k)$. Then we let $\Phi_{\b+1}\sub H(\k)$ canonically code the satisfaction relation for the structure $$(H(\k); \Phi_\b, \mtcl P_\b, \Vdash_\b^*)$$

Finally, if $\b<\k$ is a nonzero limit ordinal, we let $\Phi_\b$ be a subset of $H(\k)$ canonically coding $(\Phi_\a\,:\,\a<\b)$.

We will assume that the definition of $(\Phi_\b\,:\,\b<\k)$ is uniform in $\b$.

Finally, we define $\mtcl P_{\k}=\bigcup_{\b<\k}\mtcl P_\b$.

\section{Proving Theorem \ref{mainthm-0}}\label{relevant_facts}

We will now prove the relevant lemmas that, together, will show $\mtcl P_\k$ to witness Theorem \ref{mainthm-0}.

Given partial orders $\mtbb P$ and $\mtbb Q$, we will say that $\mtbb P$ is a \emph{weak suborder of $\mtbb Q$} in case $\dom(\mtbb P)\sub\dom(\mtbb Q)$ and for all $p_0$, $p_1\in\dom(\mtbb P)$, if $p_1\leq_{\mtbb P} p_0$, then $p_1\leq_{\mtbb Q}p_0$. Thus, $\mtbb P$ is a suborder of $\mtbb Q$ in case it is a weak suborder of $\mtbb Q$ and for all $p_0$, $p_1\in\dom(\mtbb P)$ we have that if $p_1\leq_{\mtbb Q}p_0$, then $p_1\leq_{\mtbb P} p_0$.

It is clear that if $\mtbb P$ is a weak suborder of $\mtbb Q$, then every $\mtbb P$-name is itself also a $\mtbb Q$-name.

Our first two lemmas are obvious.

\begin{lemma}\label{suborders} For all $\a<\b\leq\k$, $\mtcl P_\a$ is a weak suborder of $\mtcl P_\b$.\footnote{This lemma shows, in particular, that for all $\a<\b$, every $\mtcl P_\a$-name is also a $\mtcl P_\b$-name, and hence that our construction $(\mtcl P_\b\,:\, \b\leq\k)$ is well-defined.}
 \end{lemma}

On the other hand, it is not true in general that for all $\a<\b$, $\mtcl P_\a$ is a suborder of $\mtcl P_\b$.\footnote{In fact, s.\ Remark \ref{pathology}.}

\begin{lemma}\label{definability-0} For every $\b<\k$, $\mtcl P_\b$ and $\Vdash_\b^*$ are uniformly (in $\b$) definable over the structure $(H(\k); \in, \Phi_{\b+1})$ without parameters. \end{lemma}

Given partial orders $\mtbb P$ and $\mtbb Q$, we will say that \emph{$\mtbb P$ is a weak complete suborder of $\mtbb Q$} in case $\mtbb P$ is a weak suborder of $\mtbb Q$ and every predense subset of $\mtbb P$ is also predense in $\mtbb Q$ (i.e., if $D\sub\mtbb P$ is predense in $\mtbb P$, then for every $q\in\mtbb Q$ there are $p\in D$ and $r\in\mtbb Q$ such that $r\leq_{\mtbb Q} p$ and $r\leq_{\mtbb Q}q$). Also, we will call a sequence $\la\mtbb P_\a\,:\,\a\leq\l\rangle$ of forcing notions a \emph{weak forcing iteration} if for all $\a<\b$, $\mtbb P_\a$ is a weak complete suborder of $\mtbb P_\b$.

Given partial orders $\mtbb P$ and $\mtbb Q$ such that $\mtbb P$ is a weak suborder of $\mtbb Q$, we call a function $\p:\mtbb Q\into \mtbb P$ a \emph{weak projection of $\mtbb Q$ onto $\mtbb P$} in case for every $q\in\mtbb Q$ and every condition $p\in\mtbb P$ such that $p\leq_{\mtbb P} \p(q)$ there is some $r\in \mtbb Q$ such that $r\leq_{\mtbb Q} p$ and $r\leq_{\mtbb Q} q$. In this situation $\mtbb P$ is clearly a weak complete suborder of $\mtbb Q$.

Our sequence $(\mtcl P_\b\,:\,\b\leq\k)$ is a weak forcing iteration. In fact, given $\a<\b\leq\k$, the function sending $q\in\mtcl P_\b$ to $q\av_\a$ is a weak projection of $\mtcl P_\b$ onto $\mtcl P_\a$. This is an immediate consequence of the following lemma, the proof of which is straightforward thanks to clause (3) in the definition of the extension relation $\leq_\a$.

\begin{lemma}\label{compl-0}
Let $\a<\b\leq\k$, let $q\in\mtcl P_\b$ and $r\in\mtcl P_\a$, and suppose $r\leq_\a q\av_\a$. Then $$(F_q  \cup F_r,  \mtcl G_q\cup\mtcl G_r)$$
is a condition in $\mtcl P_\b$ extending both $q$ and $r$ in $\mtcl P_\b$.
\end{lemma}

Given $\a<\b\leq\k$, $q\in\mtcl P_\b$, and $r\in\mtcl P_\a$ extending $q\av_\a$, we write $q\oplus r$ to denote the common extension $$( F_q \cup F_r,  \mtcl G_q\cup\mtcl G_r)$$
 of $q$ and $r$ defined in the statement of Lemma \ref{compl-0}.

Given an edge $\{(M_0, \g_0), (M_1, \g_1)\}$, we will write $$\la \{ (M_0, \g_0), (M_1, \g_1)\}\ra$$ to denote the $\sub$-least set of edges containing $\{(M_0, \g_0), (M_1, \g_1)\}$ and closed under restrictions, i.e, the set $$\{\{(M_0, \a_0), (M_1, \a_1)\}\,:\,\a_0\in M_0\cap (\g_0+1), \a_1\in M_1\cap (\g_1+1)\}$$

 \begin{remark}\label{pathology} As we have just seen, our construction is a weak forcing iteration, and in fact, given any $\a<\b\leq\k$, the function sending $q\in\mtcl P_\b$ to $q\av_\a$ is a weak projection of $\mtcl P_\b$ onto $\mtcl P_\a$. However, it is not an iteration in the usual sense. Actually, it is easy to find ordinals $\a<\b$ and conditions $q_0$, $q_1\in\mtcl P_\a$ such that $q_1\leq_\b q_0$ and yet $q_0$ and $q_1$ are actually incompatible in $\mtcl P_\a$. For example, for some high enough $\b$, we can consider $\mtcl P_\b$-conditions
$q_0= (\emptyset, \mtcl G_0)$ and $q_1=(\emptyset, \mtcl G_1)$, where
\begin{itemize}
\item $\mtcl G_0=\la \{(N_0, \r_0), (N_1, \r_1)\}\ra$,
\item $\mtcl G_1$ is the union of
\begin{itemize}
\item $\mtcl G_0$,
\item $\la\{(M, \r_0)\}\ra$ and
\item $\{\{(\Psi_{N_0, N_1}(M), \g)\}\,:\,\g\in \Psi_{N_0, N_1}(M)\cap\r_1\}$,
\end{itemize} \end{itemize}
\noindent and where $\r_0<\r_1$, $M\in N_0$, $(M, \r_0)$ is a model with marker, and $\Psi_{N_0, N_1}(\r_0)>\r_1$. Let $\a=\r_1$. Then $q_1\leq_\b q_0$ but $q_0$ and $q_1$
 are incompatible in $\mtcl P_\a$ since every $r\in\mtcl P_\a$ such that $r\leq_\a q_0$, $q_1$ would have to be such that $M\in \mtcl N^r_{\r_0}$ (since it would extend $q_1$) and $M\notin \mtcl N^r_{\r_0}$ (since it would extend $q_0$ and since $\Psi_{N_0, N_1}(\r_0)>\r_1$).
 \end{remark}


The following lemma will be used in the proofs of Lemmas \ref{properness2-0} and \ref{fnr-0}.

\begin{lemma}\label{transfer}
Let $\b<\k$ and $q\in\mtcl P_\b$. Suppose $\{(N_0, \r_0), (N_1, \r_1)\}\in\mtcl G_q$, $\a\in N_0\cap \r_0$, $\dot a\in N_0$ is a $\mtcl P_\a$-name, $\varphi(x)$ is a formula in the language of set theory, $(q\restr N_0)\av_\a\in\mtcl P_\a$, and $(q\restr N_0)\av_\a\Vdash_\a \varphi(\dot a)$. Suppose $\a^*:=\Psi_{N_0, N_1}(\a)<\r_1$. Then $\Psi_{N_0, N_1}((q\restr N_0)\av_\a)=(q\restr N_1)\av_{\a^*}\in\mtcl P_{\a^*}$, $\Psi_{N_0, N_1}(\dot a)$ is a $\mtcl P_{\a^*}$-name, and $(q\restr N_1)\av_{\a^*}\Vdash_{\a^*}\varphi(\Psi_{N_0, N_1}(\dot a))$.
\end{lemma}

\begin{proof}
By Lemma \ref{definability-0} and since $$\Psi_{N_0, N_1}:(N_0; \in, \Phi_{\a+1})\into (N_1; \in, \Phi_{\a^*+1})$$ is an isomorphism, we have that $\Psi_{N_0, N_1}((q\restr N_0)\av_\a)$ is a $\mtcl P_{\a^*}$-condition and $\Psi_{N_0, N_1}(\dot a)$ is a $\mtcl P_{\a^*}$-name. And since $(q\restr N_0)\av_\a\Vdash_\a\varphi(\dot a)$, we also have that $$\Psi_{N_0, N_1}((q\restr N_0)\av_\a)\in\mtcl P_{\a^*}$$ and $$\Psi_{N_0, N_1}((q\restr N_0)\av_\a)\Vdash_{\a^*} \varphi(\Psi_{N_0, N_1}(\dot a))$$ again by Lemma \ref{definability-0} and the fact that $$\Psi_{N_0, N_1}:(N_0; \in, \Phi_{\a+1})\into (N_1; \in, \Phi_{\a^*+1})$$ is an isomorphism. Finally, clause (6) in the definition of condition, and the closure of $\mtcl G_q$ under copying, together entail that $$\Psi_{N_0, N_1}((q\restr N_0)\av_\a)=(q\restr N_1)\av_{\a^*}.$$
\end{proof}

\subsection{Properness and $\aleph_2$-c.c.}\label{subsection-properness-0} The goal of this subsection is to show both the properness and the $\al_2$-chain condition of all members $\mtcl P_\b$ of our construction.
 Our first lemma shows, given a $\mtcl P_\b$-condition $q$ and an edge $\{(N_0, \r_0), (N_1, \r_1)\}$ below $\b$ such that $q\in N_0\cap N_1$, how to add $\{(N_0, \r_0), (N_1, \r_1)\}$ to $q$.

\begin{lemma}\label{properness1-0}
Let $\b<\k$, $q\in \mtcl P_\b$, and let $\{(N_0, \r_0), (N_1, \r_1)\}$ be an edge below $\b$ such that $q\in N_0\cap N_1$. Let $\mtcl G^*$ be the union of $\mtcl G_q$ and $\la\{(N_0, \r_0), (N_1, \r_1)\}\ra$. Then $q^*=(F_q, \mtcl G^*)$ is a condition in $\mtcl P_\b$ extending $q$.
\end{lemma}

\begin{proof}
This is immediate since $\mtcl G^*$ is the $\sub$-minimal sticky set of edges closed under restrictions and such that $\mtcl G_q\cup\{(N_0, \r_0), (N_1, \r_1)\}\sub\mtcl G^*$.
\end{proof}

The proof of the following lemma is the same as that of the previous lemma.

\begin{lemma}\label{properness1-0+}
Let $\b^*\leq\k$, $q\in \mtcl P_\b$, and $N\elsub H(\k)$ such that $N\in\mtcl T_{\b+1}$ for every $\b\in N\cap\b^*$. Suppose $q\in N$. Then there is an extension $q^*\in\mtcl P_{\b^*}$ of $q$ such that $\{(N, \b)\}\in\mtcl G_{q^*}$ for every $\b\in N\cap\b^*$.
\end{lemma}

It will be convenient to prove the $\al_2$-chain condition and our main properness result in the same lemma, by a simultaneous induction. This will be the content of Lemma \ref{properness2-0}. Before getting there, it will be useful to introduce some pieces of notation and some technical lemmas.

The following lemma, which is immediate, asserts a useful interpolation property of the extension relation.

\begin{lemma}\label{restr-extend}
Let $\b<\k$, $q\in\mtcl P_\b$, and $N\in\mtcl N^q_0$. Suppose $q\restr N\in\mtcl P_\b$, and let $p\in\mtcl P_\b\cap N$ be a condition such that $q\leq_\b p$. Then $q\leq_\b q\restr N$ and $q\restr N\leq_\b p$.
\end{lemma}



\begin{lemma}\label{restr-fine}
Let $\b<\k$, $q\in\mtcl P_\b$, and $N\in\mtcl N^q_\b$.
Then $q\restr N\in\mtcl P_\b$.
\end{lemma}

\begin{proof}
We prove, by induction on $\a\leq\b$, that $$(q\restr N)\av_\a:=((F_q\restr\restr N)\restr\a, (\mtcl G_q\cap N)\av_\a)$$ is a condition in $\mtcl P_\a$.

Clause (1) in the definition of condition holds for $(q\restr N)\av_\a$ due to the fact that if $\mtcl N$ is a symmetric system and $M\in\mtcl N$, then $\mtcl N\cap M$ is also a symmetric system. Clauses (2), (6) and (7) are trivial, and clause (3) follows from the induction hypothesis. All subclauses in (4) except for (4)(b) are trivial. Finally, (4)(b) holds by clause (a) in the definition of $N$-saturatedness below $\b$ together with Lemma \ref{restr-extend},
and (5) holds by clause (b) in the definition of $N$-saturatedness below $\b$ together with, again, Lemma
\ref{restr-extend}.
\end{proof}

We will also need the following technical lemma, which is an immediate consequence of Lemma \ref{amalg-0+}.

\begin{lemma}\label{clause-1}
Let $\a<\b<\k$, $q\in\mtcl P_\b$, $N\in\mtcl N^q_0$, $t\in \mtcl P_\b\cap N$, and suppose $q\restr N\in\mtcl P_\b$ and $t\leq_\b q\restr N$.\footnote{The hypothesis that  $q\restr N\in\mtcl P_\b$ is actually not needed; if we drop it, then $t\leq_\b q\restr N$ needs to be replaced by a hypothesis to the effect that the relevant forms of clauses (1) and (2) in the definition of $\leq_\b$ hold between $t$ and $q\restr N$.} Suppose for every $Q\in\mtcl N^{\D(\mtcl G_q)}\cap N$, $Q\cap\mtcl G_t=Q\cap\mtcl G_q$. Let $p\in\mtcl P_\a$, and suppose $p\leq_\a q\av_\a$ and $p\leq_\a t\av_\a$. Let $q'=q\oplus p$ and let $\mtcl G=\mtcl G_{q'}\oplus\mtcl G_t$.
Then $\mtcl G$ is a
sticky
set of edges closed under restrictions and under copying and such that $\mtcl N^{\D(\mtcl G)}_0$ is a $\Phi_0$-symmetric system and $\mtcl N^{\D(\mtcl G)}_{\a+1}$ is a $\Phi_{\a+1}$-symmetric system for every $\a<\b$.
\end{lemma}

\begin{proof}
This is by an application of Lemma \ref{amalg-0+} with $\mtcl G_{q'}$ and $\mtcl G_{t'}$, where $t'=t\oplus (p\restr N)$.
\end{proof}

Given a set $\mtcl G$ of edges and a pertinent function $F$ such that $\dom(F)\sub\bigcup\dom(\D(\mtcl G))$, we define \emph{the closure of $F$ via edges coming from $\mtcl G$} to be the function $F^*$ with domain the set $X$ of ordinals of the form $\Psi_{\vec{\mtcl E}}(\a)$, for some $\a\in \dom(F)$ and some connected $\mtcl G$-thread $\langle \a, \vec{\mtcl E}\rangle$, defined by letting $F^*(\bar\a)$ be, for every $\bar\a\in X$, the ordered pair $(b^{F^*}_{\bar\a}, d^{F^*}_{\bar\a})$, where:

\begin{itemize}
\item $b^{F^*}_{\bar\a}=b^F_{\bar\a}\cup b^{F'}_{\bar\a}$,\footnote{Recall that $b^F_{\bar\a}$ is defined to be $\emptyset$ if $\bar\a\notin \dom(F)$. And a similar remark applies to the next bullet point.} where $b^{F'}_{\bar\a}$ is the union of the collection of sets of the form $\Psi_{\vec{\mtcl E}}``b^F_\a$, for some $\a\in \dom(F)$ and some connected $\mtcl G$-thread $\langle \a, \vec{\mtcl E}\rangle$ with $\bar\a=\Psi_{\vec{\mtcl E}}(\a)$;\footnote{$\Psi_{\vec{\mtcl E}}``b^F_\a$ is of course $b^F_\a\restr\min(\d_{\vec{\mtcl E}})$.
}

\item $d^{F^*}_{\bar\a}=d^F_{\bar\a}\cup d^{F'}_{\bar\a}$, where $d^{F'}_{\bar\a}$ is the union of the collection of sets of the form $\Psi_{\vec{\mtcl E}}``d^F_\a$, for some $\a\in \dom(F)$ and some connected $\mtcl G$-thread $\langle \a, \vec{\mtcl E}\rangle$ with $\bar\a=\Psi_{\vec{\mtcl E}}(\a)$.
\end{itemize}

We will denote this function $F^*$ by $\cl_{\mtcl G}(F)$.

Also, given pertinent functions $F_0$ and $F_1$ and given $\a\in \dom(F_0)\cap\dom(F_1)$, let $F_0(\a)+F_1(\a)$ denote $$(b^{F_0}_\a\cup b^{F_1}_\a, d^{F_0}_\a\cup d^{F_1}_\a).$$

We will then denote by $F_0+F_1$ the function $F$ with domain $\dom(F_0)\cup\dom(F_1)$ defined by letting
\begin{itemize}
\item $F(\a)=F_\epsilon(\a)$ for all $\epsilon\in\{0, 1\}$ and $\a\in\dom(F_\epsilon)\setminus \dom(F_{1-\epsilon})$ and
\item $F(\a)=F_0(\a)+F_1(\a)$ for all $\a\in\dom(F_0)\cap\dom(F_1)$.
\end{itemize}

Given a countable elementary substructure $N$ of $H(\k)$ and a $\mtcl P_\beta$-condition $q$, for some $\b<\k$, we will say that \emph{$q$ is potentially $(N, \mtcl P_\beta)$-generic} if and only if for every maximal antichain $A$ of $\mtcl P_\b$ such that $A\in N$ and every $q'\in\mtcl P_\b$ such that $q'\leq_\b q$ there is some $r\in A$ and some $q^*\in\mtcl P_\b$ such that $q^*\leq_\b r$ and $q^*\leq_{\b^\dag} q'$ for some $\b^\dag\geq\b$.  Note that this, even in the stronger version in which $\b^\dag$ is required to be $\b$, is more general than the standard notion of $(N, \mtbb P)$-genericity, for a forcing notion $\mtbb P$, which applies only if $\mtbb P\in N$. Indeed, in our situation $\mtcl P_\beta$ is of course never a member of $N$ if $N\sub H(\k)$.

We are now ready to prove the main lemma in this subsection.

\begin{lemma}\label{properness2-0} The following holds for every $\b\leq\k$.

\begin{enumerate}
\item $\mtcl P_\b$ is $\al_2$-Knaster.
\item If $\b<\k$, then for every $q\in\mtcl P_\b$ and $N\in \mtcl N^q_\b\cap\mtcl T_{\b+1}$, $q$ is potentially $(N, \mtcl P_\b)$-generic.
\end{enumerate}
\end{lemma}

\begin{proof} We prove (1) and (2) by simultaneous induction on $\b<\k$.

We start with the proof of (1). We prove that if $(q_\n\,:\,\n<\o_2)$ is a sequence of $\mtcl P_\b$-conditions, then there is $I\in [\o_2]^{\al_2}$ such that $q_{\n_0}$ and $q_{\n_1}$ are compatible in $\mtcl P_\b$ for all $\n_0$, $\n_1\in I$. Let $M_\n^*$ be, for each $\n<\o_2$, a countable elementary submodel of $H(\k^+)$ such that $\vec\Phi_\b$, $q_\n\in M_\n^*$ and let $M_\n=M_\n^*\cap H(\k)$.

By $\CH$, we may find $I\in [\o_2]^{\al_2}$ and some countable $R$ such that $M_{\n_0}\cap M_{\n_1}=R$ for all distinct $\n_0$, $\n_1$ in $I$. Again by $\CH$, and after shrinking $I$ if necessary, we may assume in addition that, for some $n$, $m<\o$, there are, for all $\nu\in I$, enumerations $(N^{\nu}_i\,:\, i<n)$ and $(\x^\nu_j\,:\,j<m)$ of $\mtcl N^{q_\nu}_0$ and $\dom(F_{q_\nu})$, respectively, such that
for all $\nu_0\neq \nu_1$ in $I$ there is an isomorphism $\Psi$ between $\mtcl M_{\n_0}$ and $\mtcl M_{\n_1}$ fixing $M_{\n_0}\cap M_{\n_1}$, where, given any $\n\in I$, $\mtcl M_\n$ is some canonically chosen structure with universe $M_\n$ coding
$R$, $(N^{\nu}_i\,:\, i<n)$, $\mtcl G_{q_\n}$, $(\x^\nu_j\,:\,j<m)$, $((b^{q_\n}_{\x^\n_j},  d^{q_\n}_{\x^\n_j})\,:\, j<m)$, and $\vec\Phi_\b\cap M_{\nu}$.

We may moreover assume that $(\a_{\n_0}; \in, \p_{\nu_0}``R)\cong (\a_{\n_1}; \in, \pi_{\nu_1}``R)$, where $\a_{\n_i}\in\o_1$ is the Mostowski collapse of $M_{\n_i}\cap \Ord$ and $\pi_{\nu_i}$ is the corresponding collapsing function. But then we have that $\Psi$ is the identity on $R\cap\Ord$. This yields that $\Psi$ is the identity on $R\cap H(\k)$ since the function $\Phi:\kappa \into H(\kappa)$ is surjective.

Let us now pick $\n_0\neq\n_1$ in $I$. We will prove that $$q^*:=((F_{q_{\n_0}}+F_{q_{\n_1}}), (\mtcl G_{q_{\n_0}}\oplus \mtcl G_{q_{\n_1}})\cup \la \{(M_{\n_0}, \b), (M_{\n_1}, \b)\}\ra)$$ is a condition in $\mtcl P_\b$ extending both $q_{\n_0}$ and $q_{\n_1}$. For this, we will prove, by induction on $\a\leq\b$, that $$q^*\av_\a:=((F_{q_{\n_0}}+F_{q_{\n_1}})\restr \a, (\mtcl G_{q_{\n_0}}\oplus \mtcl G_{q_{\n_1}})\av_\a\cup \la \{(M_{\n_0}, \b), (M_{\n_1}, \b)\}\ra\av_\a)$$ is a condition in $\mtcl P_\a$ such that $q^*\av_\a\leq_\a q_{\n_0}\av_\a$ and $q^*\av_\a\leq_\a q_{\n_1}\av_\a$.

 Clause (1) in the definition of $\mtcl P_\a$-condition holds thanks to Lemma \ref{amalg-1+}, together with Lemma \ref{amalg-restr-sym} in the case $\a<\b$. Clause (2) is trivial by construction of the function $F_{q_{\n_0}}+F_{q_{\n_1}}$, and (3) is true by the induction hypothesis. All subclauses of (4) except for (4)(b) are true by construction of $F_{q_0}+F_{q_1}$, and (4)(b) holds by the induction hypothesis. (6) follows from the fact that $\Psi$ is an isomorphism between $\mtcl M_{\n_0}$ and $\mtcl M_{\n_1}$, and (7) is immediate from the construction of $q^*$ and the present induction hypothesis.

Finally, for clause (5), suppose $\a=\a_0+1$. It is enough to prove that if $N\in\mtcl N^{q_{\n_0}}_\a$, $\Xi^{q^*\av_{\a_0+1}, \a_0}_{\d_N}\neq\emptyset$, $a\in N$, and $q\in\mtcl P_{\a_0}$ is such that $q\leq_{\a_0} q^*\av_{\a_0}$, then there is some $q'\leq_{\a_0} q$ and some $M\in\mtcl N^{q'}_{\a_0}\cap \mtcl T_{\a_0+1}\cap N$ such that $a\in M$ and $q'\Vdash_{\a_0}\d_M\notin
 \bigcup\{\dot C^{\bar\a}_{\d_N}\,:\,\bar\a\in\Xi^{q^*\av_{\a_0+1}, \a_0}_{\d_N}\}$.

 We may assume that $\a_0\in M_{\n_0}$ (the proof when $\a_0\in M_{\n_1}$ is completely symmetrical to the proof in the present case). Let us first consider the case when $\a_0\leq\Psi(\a_0)$. Let $q'\leq_{\a_0} q$ and $M\in\mtcl N^{q'}_{\a_0}\cap \mtcl T_{\a_0+1}\cap N$ such that $a\in M$ and $$q'\Vdash_{\a_0}\d_M\notin
 \bigcup\{\dot C^{\bar\a}_{\d_N}\,:\,\bar\a\in\Xi^{(q_{\n_0})\av_{\a_0+1}, \a_0}_{\d_N}\}$$ Such $q'$ and $M$ exist since, if $\Xi^{q^*\av_{\a_0+1}, \a_0}_{\d_N}\setminus \Xi^{(q_{\n_0})\av_{\a_0+1}, \a_0}_{\d_N}\neq\emptyset$, then we have that $\Xi^{(q_{\n_1})\av_{\Psi(\a_0)+1}, \Psi(\a_0)}_{\d_N}\neq\emptyset$ (since $\a_0\leq\Psi(\a_0)$), and therefore $\Xi^{(q_{\n_0})\av_{\a_0+1}, \a_0}_{\d_N}\neq\emptyset$ as $\Psi$ is an isomorphism between $\mtcl M_{\n_0}$ and $\mtcl M_{\a_1}$. Let $\bar\a\in \Xi^{q^*\av_{\a_0+1}, \a_0}_{\d_N}\setminus \Xi^{(q_{\n_0})\av_{\a_0+1}, \a_0}_{\d_N}$. We will be done in this case if we can show that $q'\Vdash_{\a_0}\d_M\notin \dot C^{\bar\a}_{\d_N}$. Let $\a_*=\Psi^{-1}(\bar\a)$ and let us note that $\a_*   \leq \a_0$ since $\bar\a\leq\Psi(\a_0)$. Since also $\a_*\in\Xi^{(q_{\n_0})\av_{\a_0+1}, \a_0}_{\d_N}$, we have that $q'\Vdash_{\a_0}\d_M\notin \dot C^{\a_*}_{\d_N}$. Suppose now that $\a_*\leq \bar\a$ (the case $\bar\a<\a_*$ is proved similarly, by reversing the roles of $M_{\n_0}$ and $M_{\n_1}$ in the following argument). Now we note that $\{(M_{\n_0}, \a_*), (M_{\n_1}, \bar\a)\}\in\mtcl G_{q'}$ and therefore, by (2) of our induction hypothesis for $\bar\a$, $q'\av_{\bar\a}$ is potentially $(M_{\n_1}, \mtcl P_{\bar\a})$-generic. Hence, for every $\x<\d_N$, every $r\leq_{\bar\a} q'$ is $\mtcl P_{\bar\a^\dag}$-compatible, for some $\bar\a^\dag\geq\bar\a$, with some condition in $M_{\n_1}$ deciding whether or not $\x\in\dot C^{\bar\a}_{\d_N}$.
 
 \begin{claim}\label{equality} 
 $q'\Vdash_{\a_0}\dot C^{\a_*}_{\d_N}=\dot C^{\bar\a}_{\d_N}$. 
 \end{claim}

\begin{proof}
Let $r\leq_{\bar\a} q'$, $\x<\d_N$,  suppose $r\Vdash_{\a_0}\x\in \dot C^{\bar\a}_{\d_N}$, and let us show that  $r\not\Vdash_{\a_0}\x\notin\dot C^{\a_*}_{\d_N}$ (arguing symmetrically we can show that if $r\Vdash_{\a_0}\x\notin \dot C^{\bar\a}_{\d_N}$, then $r\not\Vdash_{\a_0}\x\in\dot C^{\a_*}_{\d_N}$). Let $s\in M_{\n_1}$ be a $\mtcl P_{\bar\a^\dag}$-condition, for some $\bar\a^\dag\geq\bar\a$, which is compatible with $r$ in $\mtcl P_{\bar\a^\dag}$ and decides whether or not $\x\in\dot C^{\bar\a}_{\d_N}$. Since obviously also $r\Vdash_{\bar\a^\dag}\x\in\dot C^{\bar\a}_{\d_N}$, we must have that $s\Vdash_{\bar\a^\dag}\x\in\dot C^{\bar\a}_{\d_N}$, and since $\dot C^{\bar\a}_{\d_N}$ is a $\mtcl P_{\bar\a}$-name, we in fact have that $s\av_{\bar\a}\Vdash_{\bar\a}\x\in\dot C^{\bar\a}_{\d_N}$. Let $q''$ be a common extension of $r\av_{\bar\a}$ and $s\av_{\bar\a}$ in $\mtcl P_{\bar\a}$.  Since  $\{(M_{\n_0}, \a_*), (M_{\n_1}, \bar\a)\}\in\mtcl G_{q''}$, $q''$ extends $\Psi_{N_0, N_1}(s\av_{\bar\a})$. But $\Psi_{N_0, N_1}(s\av_{\bar\a})\Vdash_{\a_*}\x\in \dot C^{\a_*}_{\d_N}$ by Lemma  \ref{transfer}, from which it follows that $q''\Vdash_{\a_*}\x\in \dot C^{\a_*}_{\d_N}$. Since $q''\av_{\a_*}\leq_{\a^*}r\av_{\a^*}$, we in particular have that $r\av_{\a^*}\not\Vdash_{\a^*}\x\notin\dot C^{\a_*}_{\d_N}$, and therefore $r\not\Vdash_{\a_0}\x\notin\dot C^{\a_*}_{\d_N}$ (if $r\Vdash_{\a_0}\x\notin\dot C^{\a_*}_{\d_N}$, then we would have that also $r\av_{\a_*}\Vdash_{\a_*}\x\notin\dot C^{\a_*}_{\d_N}$ since $\dot C^{\a_*}_{\d_N}$ is a $\mtcl P_{\a_*}$-name).
 \end{proof}
 
 The above claim finishes the proof in this case since $q'\Vdash_{\a_0}\d_M\notin \dot C^{\a_*}_{\d_N}$.

The second case is when $\Psi(\a_0)<\a_0$. Since we may of course assume that $\Xi^{q^*\av_{\a_0+1}, \a_0}_{\d_N}\setminus \Xi^{(q_{\n_0})\av_{\a_0+1}, \a_0}_{\d_N}\neq\emptyset$, we in fact have that  $\Xi^{q^*\av_{\a_0+1}, \Psi(\a_0)}_{\d_N}\setminus \Xi^{(q_{\n_0})\av_{\a_0+1}, \a_0}_{\d_N}\neq\emptyset$, so it makes sense to define $\alpha_1$ as the maximum ordinal in $\Xi^{(q_{\n_1})\av_{\a_0+1}, \Psi(\a_0)}_{\d_N}$. 

Since $\Xi^{q^*\av_{\a_0+1}, \a_0}_{\d_N}\setminus \Xi^{(q_{\n_0})\av_{\a_0+1}, \a_0}_{\d_N}\neq\emptyset$, there is some $\g\in R$ such that $(\d_N, \g)$ is $\mtcl G_{q_{\n_0}}$-accessible from $(\d_N, \a_0)$ and $\mtcl G_{q_{\n_1}}$-accessible form $(\d_N, \a_1)$. Using suitable instances of the shoulder axiom as in the proof of Lemma \ref{amalg-0+} we may then find sequences 
$$\vec{\mtcl E}_0=(\la (N^{i, 0}_0, \r^{i, 0}_0), (N^{i, 0}_1, \r^{i, 0}_1)\ra\,:\, i\leq n_0)$$ and $$\vec{\mtcl E}_1=(\la (N^{i, 1}_0, \r^{i, 1}_0), (N^{i, 1}_1, \r^{i, 1}_1)\ra\,:\, i\leq n_1)$$ such that $\la\a_0, \vec{\mtcl E}_0\ra$ is a connected $\mtcl G_{q_{\n_0}}$-thread with $\Psi_{\vec{\mtcl E}_0}(\a_0)=\g$,  $\la\g, \vec{\mtcl E}_1\ra$ is a connected $\mtcl G_{q_{\n_1}}$-thread with $\Psi_{\vec{\mtcl E}_1}(\a_0)=\a_1$, $\min(\d_{\vec{\mtcl E}_0})=\d_N$, $N^{0, 0}_0=N$, and $N':=N^{1, 1}_{n_1}$ is such that $\d_{N'}=\d_N$.\footnote{Note that we can indeed proceed here as in the proof of Lemma 2.7 (more specifically, as in the verification of the shoulder axiom at the successor stages of that construction) since the definition of pertinent function implies that $\a_0$ and $\a_1$ are successor ordinals.}  
Letting then $\vec{\mtcl E}$ be the concatenation of $\vec{\mtcl E}_0$ and $\vec{\mtcl E}_1^{-1}$, we have that $\la\a_0, \vec{\mtcl E}\ra$ is a connected $\mtcl G_{q^*\av_\a}$-thread with $\Psi_{\vec{\mtcl E}}(\a_0)=\a_1$.
Since $N'\in\mtcl N^{q_{\n_1}}_{\a_1+1}$, by an instance of clause (7)(b) in the definition of condition for $q_{\n_1}$ together with Lemma \ref{compl-0}, we may find $q'\leq_{\a_0} q$ and $M'\in\mtcl N^{q'}_{\a_1}\cap \mtcl T_{\a_1+1}\cap N'$ such that $\Psi_{\vec{\mtcl E}}(a)\in M'$ and $$q'\av_{\a_1}\Vdash_{\a_1}\d_{M'}\notin
 \bigcup\{\dot C^{\bar\a}_{\d_N}\,:\,\bar\a\in\Xi^{(q_{\n_1})\av_{\a_1+1}, \a_1}_{\d_N}\}$$  Let $M=\Psi_{\vec{\mtcl E}}^{-1}(M')\in N$ and let us note that $M\in\mtcl N^{q'}_{\a_0}\cap\mtcl T_{\a_0+1}\cap N$ and $a\in M$. It thus suffices to prove that $q'\Vdash_{\a_0}\d_M\notin\dot C^{\bar\a}_{\d_N}$ for every $\bar\a\in\Xi^{q^*\av_{\a_0+1}, \a_0}_{\d_N}$. If $\bar\a\in\Xi^{(q_{\n_1})\av_{\a_0+1}, \Psi(\a_0)}_{\d_N}$, then we are clearly done since then $\bar\a\leq\a_1$. Hence, we may assume $\bar\a\in \Xi^{(q_{\n_0})\av_{\a_0+1}, \a_0}_{\d_N}\setminus \Xi^{(q_{\n_1})\av_{\a_0+1}, \Psi(\a_0)}_{\d_N}$.  Let $\a_*=\Psi(\bar\a) \leq \a_1$ and let us note that $\a_*\in\Xi^{(q_{\n_1})_{\a_1+1}, \a_1}_{\d_N}$. It thus follows that $q'\av_{\a_1}\Vdash_{\a_1}\d_M\notin\dot C^{\a_*}_{\d_N}$.  But now, arguing as in the proof of Claim \ref{equality}, using the fact that $\{(M_{\n_0}, \bar\a), (M_{\n_1}, \a_*)\}\in \mtcl G_{q'}$ and the induction hypotheses for either $\bar\a$ or $\a_*$, we get that $q'\Vdash_{\bar\a} \dot C^{\bar\a}_{\d_N}=\dot C^{\a_*}_{\d_N}$. This finishes the proof in this case since $q'\Vdash_{\a_0}\d_M\notin \dot C^{\a_*}_{\d_N}$.

Now that we know that $q^*\av_\a$ is a $\mtcl P_\a$-condition, it is easy to check that it extends both $q_{\n_0}\av_\a$ and $q_{\n_1}\av_\a$ in $\mtcl P_\a$. The only point that is not completely trivial is the verification of clause (3) in the definition of the extension relation. But this clause holds thanks to the fact that $q_{\n_0}$ and $q_{\n_1}$ carry the same information on $R$.

We will now prove (2). For this, it is enough to show that if $A\in N$ is a maximal antichain of $\mtcl P_\b$, then there is some $\b^\dag\geq\b$ such that $q$ is $\leq_{\b^\dag}$-compatible with some condition in $A\cap N$.\footnote{This is of course the same thing as showing that there is some $r^*\in A\cap N$ and some $q^*\in\mtcl P_\b$ such that $q^*\leq_\b r^*$ and $q^*\leq_{\b^\dag} q$.}
The case $\b=0$ follows at once from Lemma \ref{amalg-0}, so we will assume in what follows that $\b>0$. By extending $q$ if necessary we may, and will, assume that $q$ extends some $r_0\in A$.

Let us first consider the case that $\b=\a+1$. Suppose $\Xi^{q, \a}_{\d_N}\neq\emptyset$.
Let $\dot B$ be a $\mtcl P_\a$-name for a (partially defined) function on $\o_1 \times A$
sending $(\eta, r)$ to some condition $t\in\mtcl P_\b$ with the following properties (provided there is some such $t$; otherwise the function is not defined at $(\eta, r)$).

\begin{enumerate}
\item $t\av_\a\in \dot G_\a$
\item $t$ extends $r$.
\item $t$ extends $q\restr N$.\footnote{We note that, by the assumption that $q$ be $N$-saturated below $\b$, $q\restr N$ is actually a $\mtcl P_\a$-condition. This, however, is not an essential point; one could in fact phrase this condition alternatively, without using the fact that $q\restr N\in\mtcl P_\a$.}

\item For every $Q\in\mtcl N^t_{\a+1}$, if $\d_Q\neq\d_{Q'}$ for any $Q'\in\mtcl N^q_{\a+1}$, then $\d_Q>\eta$.
\item For every $Q\in\mtcl N^q_0\cap N$, $Q\cap\mtcl G_q=Q\cap\mtcl G_t$, $Q\cap b^t_\a=Q\cap b^q_\a$, and $Q\cap d^t_\a=Q\cap d^q_\a$.

 \end{enumerate}

By conclusion (1) for $\b$ -- which we have already proved -- we know that $\mtcl P_\b$ has the $\al_2$-c.c.\ and hence we may assume that $\dot B\in H(\k)$. Hence, by Lemma \ref{definability-0} and since $N\elsub (H(\k); \in, \Phi_{\b+1})$ and $A\in N$, we may assume that $\dot B\in N$.

 By an instance of clause (5) in the definition of $\mtcl P_\b$-condition, together with the openness of $\bar\d\setminus\dot C^{\bar\a}_{\bar\d}$ in $V^{\mtcl P_\a}$ for all $\bar\a\leq \a$ and $\bar\d<\o_1$,\footnote{Which follows from the openness of $\bar\d\setminus\dot C^{\bar\a}_{\bar\d}$ in $V^{\mtcl P_\a}$ together with Lemma \ref{compl-0}.} there is an extension $p\in\mtcl P_\a$ of $q\av_\a$ for which there are
 $M\in\mtcl N^p_\a\cap\mtcl T_{\a+1}\cap N$ and $\eta<\d_M$ such that

 \begin{enumerate}
 \item $A$, $\dot B$, $q\restr N\in M$,
 \item $p\Vdash_\a[\eta,\,\d_N]\cap \dot C^{\bar\a}_\d=\emptyset$ whenever $\bar\a$ is such that $(\d_N, \bar\a)$ is $\mtcl G_q$-accessible from $(\d_N, \a)$ and there is $(\d, \bar\d)\in b^q_{\bar\a}$ such that $\bar\d<\d_N<\d$, and
  \item $p\Vdash_\a[\eta,\,\d_M]\cap \bigcup\{\dot C^{\bar\a}_{\d_N}\,:\,\bar\a\in\Xi^{q, \a}_{\d_N}\}=\emptyset$.
 \end{enumerate}

Indeed, by openness of the relevant sets $\d\setminus\dot C^{\bar\a}_\d$ (in the extension by $\mtcl P_{\bar\a}$) we can extend $q\av_\a$ to some $p_0\in\mtcl P_\a$ deciding some $\eta_0<\d_N$ such that $[\eta_0,\,\d_N]\cap\dot C^{\bar\a}_\d$ whenever $(\d_N, \bar\a)$ is $\mtcl G_q$-accessible from $(\d_N, \a)$ and there is $(\d, \bar\d)\in b^q_{\bar\a}$ such that $\bar\d<\d_N<\d$ (since there only finitely many such pairs $(\d_N, \bar\a)$). Then, by an instance of clause (7)(b) in the definition of condition, this time using the openness of the relevant (finitely many) sets of the form $\d_N\setminus\dot C^{\bar\a}_{\d_N}$, we may extend $p_0$ to some $p\in\mtcl P_\a$ for which there is some $M\in\mtcl N^p_\a\cap\mtcl T_{\a+1}\cap N$ and some $\eta_1<\d_M$ such that $A$, $\dot B$, $q\restriction N$, $\eta_0\in M$ and such that $p\Vdash_\a[\eta_1,\,\d_M]\cap \bigcup\{\dot C^{\bar\a}_{\d_N}\,:\,\bar\a\in\Xi^{q, \a}_{\d_N}\}=\emptyset$. Then, letting $\eta=\max\{\eta_0, \eta_1\}$, we get the desired conclusion.


By (2) of the induction hypothesis for $\a$ there is some $u\in M\cap\mtcl P_\a$, $r^*\in M\cap A$, and $t^*\in M\cap\mtcl P_\b$ such that $u$ is $\mtcl P_{\a^\dag}$-compatible with $p$ for some $\a^\dag\geq\a$ and $u$ forces in $\mtcl P_\a$ that $\dot B_{\dot G_\a}(\eta, r^*)$ is defined and $\dot B_{\dot G_\a}(\eta, r^*)=t^*$. This is true since, in the extension of $V$ by $\mtcl P_\a$, the existence of such a member of $A$ is witnessed by $r_0$, as in turn witnessed by $q$, and is expressible over $(H(\k)^{V[\dot G_\a]}; \in, H(\k)^V, \dot G_\a)$ by a sentence with the objects $\dot B$, and $\eta$ as parameters, all of which are in $M$). Let also $p'\in\mtcl P_\a$ be such that $p'\leq_{\a^\dag} p$ and $p'\leq_{\a^\dag} u$.

Let $\b^\dag$ be any ordinal such that $\b^\dag\geq\b$ and such that $\Psi_{N_0, N_1}(\r_0)<\b^\dag$ for every edge $\{(N_0, \r_0), (N_1, \r_1)\}\in\mtcl G_q$. We will now construct a condition in $\mtcl P_\b$ $\leq_\b$-extending $p'$, $t^*$ and $\leq_{\b^\dag}$-extending $q$.
For this, we let $q'=q\oplus p'$, $\mtcl G^*=\mtcl G_{q'}\oplus\mtcl G_{t^*}$, and let $F^*=\cl_{\mtcl G^*}(F_{q'}+F_{t^*})$. Let $q^*=(F^*, \mtcl G^*)$. We already know that $q^*\av_\a$ is a condition in $\mtcl P_\a$, and using this fact we will show that $q^*\in\mtcl P_\b$. It will then follow that $q^*\leq_\b r^*$ (by Lemma \ref{restr-extend}, since $t^*\leq_\b r^*$ and since clearly $q^*\restr N\leq_\b t^*$) and $q^*\leq_{\b^\dag} q$ (by $t^*\leq_\b q\restr N$ together with the fact that (5) above holds for $t^*$, the definition of $\mtcl G^*$ as $\mtcl G_{q'}\oplus\mtcl G_{t^*}$, and the definition of $F^*$ as $\cl_{\mtcl G^*}(F_{q'}+F_{t^*})$, and the choice of $\b^\dag$), which will finish the proof of the lemma in this case since $r^*\in N$.

Clause (1) in the definition of condition holds for $q^*$ by Lemma \ref{clause-1} noting that, by the choice of $t^*$, we are indeed under the hypotheses of this lemma. As usual (2) is trivial, (3) follows from the fact that $q^*\av_\a\in\mtcl P_\a$, and all subclauses of (4) except for (4)(b) are trivial. (4)(b) follows from our choice of $\eta$ and the fact that $t^*$ satisfies (5) with respect to $\eta$, together with Lemma \ref{transfer} and the induction hypothesis, and (5) follows from Lemma \ref{transfer}, the induction hypothesis, and the fact that for every $Q\in\mtcl N^q_\b$ such that $\d_Q<\d_N$ and every $\bar\a\in\Xi^{q^*, \a}_{\d_Q}$ there is  some $\a^\dag\in\Xi^{q^*, \a}_{\d_Q}\cap M$ such that $q ^*\Vdash_\a \dot C^{\bar\a}_{\d_Q}=\dot C^{\a^\dag}_{\d_Q}$---by arguments as in the verification of clause (5) for the amalgamation $q^*$ in the proof of part (1), using (2) of the induction hypothesis for $\a$ and for the relevant $\bar\a$.
Finally, (6) follows from the construction of $F^*$ as $\cl_{\mtcl G^*}(F_{q'}+F_{t^*})$, and (7) is verified in the same way as (5).

The argument when $\Xi^{q, \a}_{\d_N}=\emptyset$ is exactly the same, except that in the choice of $\eta$ we make sure that it satisfies (1) and (2) above, rather than (1)--(3). Also, in this case there is no need to argue in any $M\in N$; we can work in $N$ itself.

It remains to prove the lemma in the case that $\b$ is a limit ordinal. Let $\a\in N\cap \b$ be such that $\dom(F_q)\cap [\a, \b)\cap N=\emptyset$ and let $\b^\dag$ be defined in the same was as in the successor case.
 Using (1) of the induction hypothesis for $\a$, we may then find $r^*\in A\cap N$, $t^*\in\mtcl P_\b\cap N$, $p\in\mtcl P_\a$, and $\a^\dag\geq\a$ such that
\begin{enumerate}
\item $p\leq_\a t^*\av_\a$,
\item $t^*\leq_\b r^*$,
\item $t^*\leq_\b q\restr N$,
\item $p\leq_{\a^\dag} q\av_\a$, and
\item for every $Q\in\mtcl N^q_0\cap N$, $Q\cap\mtcl G_q=Q\cap\mtcl G_{t^*}$.
\end{enumerate}

Finally, we amalgamate $p$, $q$ and $t^*$ into a condition $q^*\in\mtcl P_\b$ as in the successor case; specifically, we let $q'=q\oplus p$, $\mtcl G^*=\mtcl G_{q'}\oplus\mtcl G_{t^*}$, $F^*=\cl_{\mtcl G^*}(F_{q'}+F_{t^*})$, and $q^*=(F^*, \mtcl G^*)$. The verification that $q^*$ is a condition in $\mtcl P_\b$ such that $q^*\leq_\b t^*$ and $q^*\leq_{\b^\dag} q$ is contained in the corresponding proof in that case. Since $r^*\in N$, this concludes the proof in the present case, and hence the proof of the lemma.
\end{proof}




\begin{corollary} $\mtcl P_\k$ is proper. \end{corollary}

\begin{proof}
Let $N^*\elsub H(\k^+)$ be a countable model such that $\Phi\in N^*$ and let $q\in \mtcl P_\k\cap N^*$. It is enough to show that there is an extension $q^*\in \mtcl P_\k$ of $q$ which is $(N^*, \mtcl P_\k)$-generic. Let $N=N^*\cap H(\k)$. By Lemma \ref{properness1-0+} there is an extension $q^*\in\mtcl P_\k$ of $q$ such that $\{(N, \b)\}\in\mtcl G_{q^*}$ for every $\b\in N\cap\k$. Let now $A\in N^*$ be a maximal antichain of $\mtcl P_\k$ and let $q'\in\mtcl P_\k$ be such that $q'\leq_\k q^*$. We will show that $q'$ is $\leq_\k$-compatible with a condition in $A\cap N$.

By the $\al_2$-c.c.\ of $\mtcl P_\k$ (i.e., case $\k$ of Lemma \ref{properness2-0}(1)) and $\cf(\k)\geq\o_2$, $A\in N$ and there is some ordinal $\b\in N$ such that $A$ is also a maximal antichain of $\mtcl P_\b$. Since $A$ is a maximal antichain of $\mtcl P_\k$ to begin with, we may assume, by picking $\b$ high enough, that $\dom(F_{q'})\setminus \b=\emptyset$.
By Lemma \ref{properness2-0}(2) applied to $\b$ there are then $r^*\in A\cap N$, $q^*\in \mtcl P_\b$ and $\b^\dag\geq\b$ such that $q^*\leq_\b r^*$ and $q^*\leq_{\b^\dag} q'\av_\b$. Let $\mtcl G_{**}=\mtcl G_{q^*}\oplus \mtcl G_{q'}$ and $F_{**}=\cl_{\mtcl G_{**}}(F_{q^*})$ and let $q^{**}=(F_{**}, \mtcl G_{**})$. Since $\dom(F_{q'})\sub\b$, it is then easy to show, by arguing as in the proof of Lemma \ref{properness2-0}, that $q^{**}$ is a condition in $\mtcl P_\k$ such that $q^{**}\leq_\k q'$. But now we are done since also $q^{**}\leq_\k r^*$.\end{proof}

\begin{remark} Our argument to prove properness does not work for $\b<\k$. In fact it may not be the case that $\mtcl P_\b$ be proper in general for $\b<\k$. \end{remark}

\subsection{New reals}\label{subsection-new-reals-0}

The following is proved in \cite{separating-fat}, Fact 2.6.

\begin{lemma}\label{new-reals} $\mtcl P_0$ adds $\al_1$-many Cohen reals. \end{lemma}

We will now use clause (6) in the definition of condition (and the closure of $\mtcl G_q$ under copying whenever $q$ is a condition) to prove
 Lemma \ref{fnr-0}, which
 is a counterpoint to Lemma \ref{new-reals}. Lemma \ref{fnr-0} shows that $\mathcal P_\k$ does not add more than $\aleph_1$-many new reals, and hence that this forcing preserves $\CH$ (cf.\ the proof of Proposition 2.7 in \cite{separating-fat} or the proof sketched in the introduction).

\begin{lemma}\label{fnr-0} (Few new reals) $\mathcal P_\k$ adds not more than $\aleph_1$-many new reals.
\end{lemma}

\begin{proof}
Suppose, towards a contradiction, that there is a $\mathcal P_\k$-condition $q$ and a sequence $(\dot r_\nu)_{\nu<\omega_2}$ of $\mathcal P_\k$-names for subsets of $\omega$ such that $$q\Vdash_{\kappa}\dot r_\nu \neq\dot r_{\nu'}$$ for all $\nu\neq \nu'$. We will find an extension $q^*$ of $q$ together with $\n_0\neq\n_1$ such that $q^*\Vdash_\k \dot r_{\n_0}=\dot r_{\n_1}$, which will be a contradiction.

By $\mtcl P_\k=\bigcup_{\b<\k}\mtcl P_\b$, we may fix $\b<\k$ such that $q\in\mtcl P_\b$.
Let $\nu<\o_2$ be given. By Lemma \ref{properness2-0}(1) and, again, the fact that $\mtcl P_\k=\bigcup_{\b<\k}\mtcl P_\b$, we may assume that $\dot r_\nu\in H(\k)$ and we may find $\b_\nu<\k$ above $\b$ and such that $\dot r_\nu$ is in fact a $\mtcl P_{\b_\nu}$-name for a subset of $\o$.

For each $\nu<\omega_2$ let $N^\ast_{\nu}\elsub H(\k^+)$ be countable and containing $q$, $\Phi$, $\dot r_\nu$, and $\b_{\nu}$,
and let $N_{\nu}=N^\ast_{\nu}\cap H(\kappa)$.

Using $\CH$ we may find $\n_0\neq \n_1$ in $\o_2$ such that $$(N_{\nu_0}; \in, q, \dot r_{\nu_0}, \{\b_{\n_0}\}, \Phi_{\b_{\n_0}+1})$$ and $$(N_{\nu_1}; \in, q, \dot r_{\nu_1},  \{\b_{\n_1}\}, \Phi_{\b_{\n_1}+1})$$ are isomorphic structures. In particular, $$e=\{(N_{\n_0},  \b_{\n_0}+1), (N_{\n_1},  \b_{\n_1}+1)\}$$ is then an edge.

Let us assume that $\b_{\n_0}\geq \b_{\n_1}$. By Lemma \ref{properness1-0} we may find an extension $q^*\in\mtcl P_{\b_{\n_0}}$ of $q$ such that $e\in \mtcl G_{q^*}$ and $F_{q^*}=F_q$.
Let now $q'\in\mtcl P_{\b_{\n_0}}$ be any extension of $q^*\av_{\b_{\n_0}}$ and suppose, towards a contradiction, that $q'\Vdash_{\b_{\n_0}} n\in \dot r_{\nu_0}\D\dot r_{\nu_1}$ for some $n<\o$. Let us assume that $q'\Vdash_{\b_{\n_0}} n\in \dot r_{\nu_0}\setminus \dot r_{\nu_1}$.

By Lemma \ref{properness2-0}(2), $q^\ast\av_{\b_{\n_0}}$ is potentially $(N_{\nu_0}, \mtcl P_{\b_{\n_0}})$-generic. Hence, there are $\b_{\n_0}^\dag\geq\b_{\n_0}$ and $q''\in\mtcl P_{\b_{\n_0}}$, $q''\leq_{\b_{\n_0}^\dag} q'$, such that $q''\leq_{\b_{\n_0}} p$ for some $p\in N_{\n_0}\cap\mtcl P_{\b_0}$
such that $p\Vdash_{\b_{\n_0}} n\in\dot r_{\nu_0}$.
We know that $(q''\av_{\b_{\n_0}})\restr N_{\n_0}\in\mtcl P_{\b_{\n_0}}$ (by Lemma \ref{restr-fine}) and $(q''\av_{\b_{\n_0}})\restr N_{\n_0}\leq_{\b_{\n_0}} p$ (by Lemma \ref{restr-extend}). We then have that $$(q''\av_{\b_{\n_0}})\restr N_{\n_0}\Vdash_{\b_{\n_0}} n\in \dot  r_{\nu_0},$$ and therefore $(q''\av_{\b_{\n_1}})\restr N_{\n_1}\in\mtcl P_{\b_{\n_1}}$ and $$(q''\av_{\b_{\n_1}})\restr N_{\n_1}\Vdash_{\b_{\n_1}} n\in \Psi_{N_{\n_0}, N_{\n_1}}(\dot r_{\n_0})$$ by Lemma \ref{transfer}.
Again by Lemmas \ref{restr-fine} and \ref{restr-extend}, we have that $q''\av_{\b_{\n_1}}\leq_{\b_{\n_1}} (q''\av_{\b_{\n_1}})\restr N_{\n_1}$, and therefore $q''\av_{\b_{\n_1}}\Vdash_{\b_{\n_1}}n\in \Psi_{N_{\nu_0}, N_{\nu_1}}(\dot r_{\nu_0})$.\footnote{Cf.\ the argument in the verification of clause (5) in the definition of condition for the amalgamation $q^*$ in the proof of $\aleph_2$-c.c.\ from Lemma \ref{properness2-0}.} But this yields a contradiction since $\Psi_{N_{\nu_0}, N_{\nu_1}}(\dot r_{\nu_0})=\dot r_{\nu_1}$.

The argument in the case that $q'\Vdash_{\b_{\n_0}} n\in \dot r_{\nu_1}\setminus \dot r_{\nu_0}$ is symmetrical to the proof in the previous case; in that case, we take $r\in  N_{\nu_0}\cap\mtcl P_{\b_{\n_0}}$ such that $r\Vdash_{\b_{\n_0}} n\notin\dot r_{\nu_0}$.\footnote{Compare this proof with the proof of Claim \ref{equality}.}
\end{proof}

Given $\a<\k$ and a $\mtcl P_{\k}$-generic filter $G$, let $$D^G_\a=\{\d_N\,:\, N\in\mtcl N^G_{\a+1}\}$$ Let also $\dot D_\a$ be a $\mtcl P_{\k}$-name for $D^G_\a$.

We now prove the other conclusion in Theorem \ref{mainthm-0} involving cardinal arithmetic.

\begin{lemma}\label{power-aleph1}
$\mtcl P_\k$ forces $2^{\aleph_1}=\k$.
\end{lemma}

\begin{proof}
In order to prove that $\Vdash_{\mtcl P_\k}2^{\aleph_1}\geq\k$, it suffices to show that $\mtcl P_\k$ forces that $\dot D_{\a_0}\setminus\dot D_{\a_1}\neq\emptyset$ for all $\a_0<\a_1$. For this, let $q$ be a $\mtcl P_\k$-condition, which we may assume is such that $\a_1\in\dom(F_q)$, and let $N\in [H(\k)]^{\al_0}$ be a sufficiently correct model such that $q\in N$. By the same argument as in the proof of Lemma \ref{properness1-0} we may find an extension $q'\in\mtcl P_\k$ of $q$ such that $N\in\mtcl N^{q'}_{\a_0+1}$ and $\mtcl N^{q'}_{\a_1+1}=\mtcl N^q_{\a_1+1}$. Let $\d<\d_N$ be above $\d_M$ for every $M\in\mtcl N^q_{\a_1+1}$ and let $q^*\in\mtcl P_\k$ be the extension of $q'$ resulting from adding $(\d, \d_N)$ to $d^{q'}_{\a_1}$. Then $q^*$ forces that $\d_N\in\dot D_{\a_0}\setminus\dot D_{\a_1}$.  Since $q\in\mtcl P_\k$ was arbitrary, this density lemma shows that $\mtcl P_\k$ forces $\dot D_{\a_0}\setminus\dot D_{\a_1}\neq\emptyset$.

Finally, a simple counting argument of nice $\mtcl P_\k$-names for subsets of $\o_1$ (s.\ \cite{KUNEN}) using the $\al_2$-chain condition of $\mtcl P_\k$ and the fact that $|\mtcl P_\k|^{\al_1}=\k^{\al_1}=\k$ shows that $\mtcl P_\k$ forces $2^{\al_1}\leq\k$.
\end{proof}

\subsection{Measuring}\label{subsection-measuring-0}

The following lemma completes the proof of Theorem \ref{mainthm-0}.

\begin{lemma}\label{measuring-0} $\mathcal P_\k$ forces $\Measuring$.
\end{lemma}

\begin{proof}
Let $G$ be $\mathcal P_\k$-generic and let $\vec C=(C_\delta\,:\,\delta\in\Lim(\omega_1))\in V[G]$ be a club-sequence on $\omega_1$. We want to see that there is a club of $\omega_1$ in $V[G]$ measuring $\vec C$.
By $\mtcl P_\k=\bigcup_{\a<\k}\mtcl P_\a$ together with the $\al_2$-c.c.\ of $\mtcl P_\k$, we may assume that, for some $\a_0<\k$, $\vec C=\dot C_G$ for some $\mtcl P_{\a_0}$-name $\dot C\in H(\k)$ for a club-sequence on $\o_1$. Again by the $\al_2$-c.c.\ of $\mtcl P_\k$ and the unboundedness of $\{\a\in\Succ(\k)\,:\,\Phi(\a)=\dot C\}$ in $\k$, we may fix some $\a>\a_0$ in $\Succ(\k)$ such that $\Phi(\a)=\dot C$. We then have that $\Phi(\a)$ is a $\mtcl P_\a$-name, and by Lemma \ref{compl-0} it is in fact a $\mtcl P_\a$-name for a club-sequence on $\o_1$.
Hence, we then have that $\vec C=\Phi(\a)_G$. We will see that $(\dot D_\a)_G$ is a club of $\o_1$ measuring $\vec C$.

First of all, it is easy to see that $\dot D_\a$ is forced to be unbounded in $\o_1$. In fact, given any condition $q\in\mtcl P_\k$ and any sufficiently correct countable $N\elsub H(\k)$ such that $q$, $\a\in N$, we may find by Lemma \ref{properness1-0} an extension $q^*\in\mtcl P_\k$ of $q$ such that $N\in \mtcl N^{q^*}_{\a+1}$, and every such condition forces that $\d_N\in \dot D_\a$.

\begin{claim}\label{closure} $D^G_\a$ is closed in $\o_1$.
\end{claim}

\begin{proof}
It suffices to prove that if $\d\in\Lim(\o_1)$ and $q\in\mtcl P_\k$ are such that $q$ forces $\d$ to be a limit point of $\dot D_\a$, then there is some $N\in\mtcl N^q_{\a+1}$ such that $\d_N=\d$.

Suppose, towards a contradiction, that $q\in\mtcl P_\k$ and $\d\in\Lim(\o_1)$ are such that $q$ forces $\d$ to be a limit point of $\dot D_\a$ but there is no $N\in\mtcl N^q_{\a+1}$ such that $\d_N=\d$.
 We may extend $q$ to a condition $q'$ obtained by adding $(\bar\d, \d)$ to $d^q_\a$, where $\bar\d<\d$ is above $\d_M$ for every $M\in\mtcl N^q_{\a+1}$ such that $\d_M<\d$, and taking copies under $\Psi_{N_0, N_1}$ as dictated by relevant edges $\{(N_0, \r_0), (N_1, \r_1)\ra\}\in\mtcl G_q$. But that yields a contradiction since then $q'$ forces, by clause (4)(d) in the definition of condition, that $\dot D_\a\cap \d$ is bounded by $\bar\d$.
\end{proof}

Given any $q\in G$ such that $\alpha\in \dom(F_q)$ and any limit point $\d\in D^G_\a$, if $(\d, \bar\d)\in b^q_\a$ for some $\bar\d<\d$, then $D^G_{\alpha}\cap (\bar\d,\, \d)$ is disjoint from $C_\d$. Hence, in order to finish the proof of the lemma it is enough to show that if $q\in G$ is such that $\a\in \dom(F_q)$, $N\in\mtcl N^q_{\a+1}$, and there is no $q'\in G$ extending $q$ and such that $\d_N\in \dom(b^{q'}_\a)$, then a tail of $D^G_\a$ is contained in $C_{\d_N}$.

So, let $q$ be a condition with $\a\in \dom(F_q)$ and let $N\in\mtcl N^q_{\a+1}$ be such that $\d_N\notin \dom(b^{q'}_\a)$ for any $q'\in\mtcl P_\k$ extending $q$. It suffices to find an extension $q^*$ of $q$ in $\mtcl P_\k$ and some $\delta <\d_N$ with the property that if $q'\in\mtcl P_\k$ extends $q^*$ and $M\in\mtcl N^{q'}_{\a+1}$ is such that $\delta <\d_M<\d_N$, then $q'\av_\a\Vdash_\a \d_M\in \dot C^\a_{\d_N}$.

We will assume that $\Xi^{q\av_{\a+1}, \a}_{\d_N}\neq\emptyset$---the proof in the case $\Xi^{q\av_{\a+1}, \a}_{\d_N}=\emptyset$ is a simpler version of the proof in this case. Let
$\a_0=\max(\Xi^{q, \a}_{\d_N})$, which is well-defined since $\emptyset\neq\Xi^{q\av_{\a+1}, \a}_{\d_N}\sub\Xi^{q, \a}_{\d_N}$.
 As usual, we may find a sequence $\vec{\mtcl E}=(\la (N^i_0, \r_0^i), (N^i_1, \r^1_i)\ra\,:\, i\leq n)$ such that $\la \a, \vec{\mtcl E}\ra$ is a connected $\mtcl G_q$-thread with $\min(\d_{\vec{\mtcl E}})=\d_N$, $\Psi_{\vec{\mtcl E}}(\a)=\a_0$, $N^0_0=N$, $N^n_1\in\mtcl N^q_{\a_0+1}$, and $\d_{N^n_1}=\d_N$.

\begin{claim}\label{cl22} There is some extension $q_0\in\mtcl P_\k$ of $q$ and some $a\in N$ such that $q_0$ forces in $\mtcl P_\k$ that if $M\in \mtcl N^{\dot G_\k}_{\a_0}\cap\mtcl T_{\a_0+1}\cap N^n_1$, $\Psi_{\vec{\mtcl E}}(a)\in M$, and $$\d_M\notin \bigcup\{\dot C^{\bar\a}_{\d_N}\,:\,\bar\a\in\Xi^{q, \a_0}_{\d_N}\},$$ then $\d_M\in \dot C^\a_{\d_N}$.
\end{claim}

\begin{proof}
Let us assume that the conclusion fails. Given any extension $q'$ of $q$ and any $a\in N$, by an instance of clause (7)(b) in the definition of condition for $q\av_{\a_0+1}$ together with Lemma \ref{compl-0}, there is some $q''\leq_\k q'$ and some $M\in\mtcl N^{q''}_{\a_0}\cap\mtcl T_{\a_0+1}\cap N^n_1$ such that $\Psi_{\vec{\mtcl E}}(a)\in M$ and  $$q''\av_{\a_0}\Vdash_{\a_0}\d_M\notin \bigcup\{\dot C^{\bar\a}_{\d_N}\,:\,\bar\a\in\Xi^{q, \a_0}_{\d_N}\}$$ By our assumption, we then have that $q''\av_{\a_0}\not\Vdash_{\a_0}\d_M\in \dot C^\a_{\d_N}$. Hence, every such $q''$ forces $\d_M\notin\dot C^\a_{\d_N}$.  We have thus seen that $q$ forces that for every $a\in N$ there is some $M\in \mtcl N^{\dot G_\k}_{\a_0}\cap\mtcl T_{\a_0+1}\cap N^n_1$ such that $\Psi_{\vec{\mtcl E}}(a)\in M$ and $$\d_M\notin  \bigcup\{\dot C^{\bar\a}_{\d_N}\,:\,\bar\a\in\Xi^{q, \a_0}_{\d_N}\}\cup\{\dot C^\a_{\d_N}\}$$

Let now $\bar\d<\d_N$ be above $\d_Q$ for every $Q\in \mtcl N^q_{\a+1}$ such that $\d_Q<\d_N$ and let $q^*$ be the result of adding $(\d_N, \bar\d)$ to $b^q_\a$ and closing under relevant isomorphisms $\Psi_{N_0, N_1}$. Then $q^*$ is a condition in $\mtcl P_\k$ extending $q$ (all clauses in the definition of condition except for (7)(b) are immediate, and (7)(b) follows from $\Xi^{q^*\av_{\a+1}, \a}_{\d_N}\setminus\{\a\}=\Xi^{q\av_{\a+1}, \a}_{\d_N}\sub\Xi^{q, \a}_{\d_N}$ and the property of $q$ we have just proved), which is a contradiction since $\d_N\in \dom(b^{q^*}_\a)$.
\end{proof}

Let $q_0$ and $a\in N$ be as in Claim \ref{cl22}. Let $\d<\d_N$ be above $\d_Q$ for every $Q\in\mtcl N^{q_0}_{\a+1}$ such that $\d_Q<\d_N$ and let $q^*$ be the extension obtained by adding the pair $(\d, a)$ to $d^{q_0}_\a$ and closing under relevant isomorphisms $\Psi_{N_0, N_1}$.

We now show that $q^*$ and $\d$ are as desired. For this, suppose $q'\in\mtcl P_\k$ extends $q^*$ and $M\in\mtcl N^{q'}_{\a+1}$ is such that $\d<\d_M<\d_N$. By an instance of (4)(d) in the definition of condition for $q'$, we then have some $M'\in\mtcl N^{q'}_{\a+1}$ such that $\d_{M'}=\d_M$ and $a\in M'$. By the shoulder axiom for $\mtcl N^{q'}_{\a+1}$ there is some $N'\in\mtcl N^{q'}_{\a+1}$ such that $\d_{N'}=\d_N$ and $M'\in N'$. Then $M''=\Psi_{N', N}(M')\in \mtcl N^{q'}_{\a+1}\cap N$ and $a\in M''$ since $\Psi_{N', N}(a)=a$ as $a\in N\cap N'$. Since $M''\in\mtcl N^{q'}_{\a+1}\cap N$, we then have of course that $$q'\av_\a\Vdash_\a\d_{M''}\notin \bigcup\{\dot C^{\bar\a}_{\d_N}\,:\,\bar\a\in \Xi^{q, \a_0}_{\d_N}\},\footnote{Note the presence in this expression of $\Xi^{q, \a_0}_{\d_N}$ rather than $ \Xi^{q', \a_0}_{\d_N}$ or $\Xi^{q'\av_{\a+1}, \a}_{\d_N}$.}$$ from which it follows by the choice of $a$ that $q'\av_\a\Vdash_\a\d_{M''}\in\dot C^\a_{\d_N}$. This finishes the proof since $\d_{M''}=\d_M$.
\end{proof}

\subsection{On adapting the construction of Theorem \ref{mainthm-0-intro} to other contexts}\label{conclusion}

It will be sensible to finish this section with some words addressing the issue of what goes wrong if we try to modify the present forcing so as to force $\CH$ together with $\Unif(\vec C)$, for some given ladder system $\vec C=(C_\delta\,:\,\delta\in\Lim(\omega_1))$---as we mentioned in the introduction, the conjunction of these two statements cannot hold. One could in fact try to build something like a sequence of partial orders $(\mathcal P_\beta)_{\beta\leq\k}$ in our construction in such a way that, at every stage $\alpha<\k$, we attempt to add a uniformizing function on $\vec C$ for some colouring $F:\Lim(\omega_1)\longrightarrow \{0, 1\}$ fed to us by our book-keeping function $\Phi$. Thus, rather than the present pairs $(b, d)$, we would plug in conditions for a natural forcing for adding such a uniformizing function with finite conditions.

Everything would seem to go well---and in particular our construction would have the $\aleph_2$-c.c., would be proper, and would preserve $\CH$---except that, because of the copying constraint expressed in the corresponding version of clause (6) in the definition of condition, it would not be able to force $\Unif(\vec C)$. The reason is that we would not be in a position to rule out situations in which there is a condition $q$ with, for example, an edge $\{(N_0, \r_0), (N_1, \r_1)\}$ in $\mtcl G_q$ for which there is some $\alpha\in N_0\cap \r_0$ such
that the colour of $\dot F(\alpha)$ at $\delta_{N_0}$ is forced to be, say, $0$, whereas the colour of $\dot F(\bar\alpha)$ at $\delta_{N_0}$ is forced to be $1$ (where $\bar\a=\Psi_{N_0, N_1}(\a)$ and where $\dot F(\xi)$ denotes of course the name for the colouring to be uniformized at stage $\xi$ of the construction). The requirement, imposed by the current version of clause (6), that any relevant amount of information below $\d_{N_0}$ on the generic uniformizing function at the coordinate $\alpha$ be copied over to the coordinate $\bar\alpha$ would then make it impossible for these generic uniformizing functions to be defined on any tail of $C_{\delta_{N_0}}$.
This type of problems does not arise when forcing $\Measuring$ due to the more lenient nature of the `guessing' in this case; if we cannot get the club to eventually stay outside  a given $C_\d$, then it has to eventually get inside (see the density argument in the proof of Lemma \ref{measuring-0}). The fact whether one or the other is the case is determined by the specific club-sequence being measured (and by the `shape' of the surrounding  condition, of course).

It may also be worth pointing out that the type of situation des\-cribed above is a source of serious obstacles towards trying to force any reasonable forcing axiom to hold together with $\CH$ using the present me\-thods. To see this in a particularly simple case, suppose, for example, that $(\mtcl Q_\b)_{\b\leq\k}$ is exactly as our present construction $(\mtcl P_\b)_{\b\leq\k}$, except that at each stage we force with Cohen forcing. This construction enjoys all relevant nice properties that $(\mtcl P_\b)_{\b\leq\k}$ has. On the other hand, $\mtcl Q_\k$
cannot possibly force $\FA_{\al_1}(\mbox{Cohen})$, as it preserves $\CH$. Letting $\a^*<\k$ be such that all reals in $V^{\mtcl Q_{\k}}$ have already appeared in $V^{\mtcl Q_{\a^*}}$, if $\a<\k$ is above $\a^*$, then the real constructed by the generic at the coordinate $\a$ will actually fail to be Cohen-generic over $V^{\mtcl Q_{\a^*}}$; in fact, for every condition $q\in\mtcl Q_\k$ such that $\a\in\dom(F_q)$ there will be a condition $q'$ extending $q$ for which there is connected $\mtcl G_{q'}$-thread $\la\a, \vec{\mtcl E}\ra$ such that $\bar\a:=\Psi_{\vec{\mtcl E}}(\a)<\a^*$.
 The information at the coordinate $\bar\a$ contained in any extension of $q'$ will then have to be  copied over into the coordinate $\a$, which in this si\-tu\-ation means that the real $r_\a$ constructed at the coordinate $\a$ is identical to the real at $\bar\a$, and this obviously prevents $r_\a$ from being Cohen-generic over $V^{\mtcl Q_{\a^*}}$.

\end{document}